\documentclass{article}

\usepackage[margin=1in]{geometry}
\usepackage{standalone}
\usepackage{hyperref}
\usepackage{cite}
\usepackage{microtype}
\usepackage{enumitem}
\usepackage{subfigure}
\usepackage{soul}
\usepackage{placeins}
\usepackage{authblk}
\usepackage{tikz}
\usepackage{xcolor}
\usepackage{pgfplots}\pgfplotsset{compat=1.18}
\usepackage{amsmath, amssymb,bm}
\usepackage{mathtools}
\usepackage{amsthm}
\usepackage{mathrsfs}
\usepackage{algorithm}
\usepackage{algpseudocode}

\theoremstyle{definition}
\newtheorem{theorem}{Theorem}

\newtheorem{remark}{Remark}
\newtheorem{assumption}{Assumption}

\definecolor{plot-red}{rgb}{0.9,0.3,0.3}
\definecolor{plot-blue}{rgb}{0.2,0.6,0.9}
\definecolor{plot-darkblue}{rgb}{0.1,0.4,0.8}
\definecolor{plot-green}{rgb}{0.3,0.8,0.3}
\definecolor{plot-purple}{rgb}{0.8,0.3,0.8}

\tikzstyle{block}=[rectangle, draw=black, fill=white, rounded corners = 3pt, inner sep = 4pt, thick, minimum width = 8mm, minimum height = 4mm]
\tikzstyle{longblock}=[rectangle, draw=black, fill=white, rounded corners = 3pt, inner sep = 4pt, thick, minimum height = 4mm, minimum width = 10mm]
\tikzstyle{datamatrix}=[rectangle, draw=black , fill=white, rounded corners = 3pt, inner sep = 4pt, thick, minimum height = 4mm, minimum width = 10mm]
\tikzstyle{add}=[circle, draw=black, fill = white, inner sep = 1pt, thick, minimum size = 4mm]
\tikzstyle{basicline}=[->, thick, color = black]
\tikzstyle{invisiblenode}=[circle,draw=none, fill=none, minimum size = 0mm, inner sep = 0]

\pgfplotsset{
    datamarkstyle/.style ={only marks, mark size = 2pt, mark options = {line width = 0.7pt}},
    datamarkstyleB/.style ={only marks, mark size = 3pt, mark options = {line width = 0.9pt}},
}

\let\originalleft\left
\let\originalright\right
\renewcommand{\left}{\mathopen{}\mathclose\bgroup\originalleft}
\renewcommand{\right}{\aftergroup\egroup\originalright}

\newcommand{\cL}{\mathcal{L}}

\newcommand{\bbR}{\mathbb{R}}
\newcommand{\bbN}{\mathbb{N}}

\newcommand{\one}{\boldsymbol{1}} 
 
\newcommand{\Id}{\mathrm{I}}

\newcommand{\E}{\ensuremath{\mathbb{E}}}

\newcommand{\ex}[1]{\ensuremath{\mathbb{E}\left[ #1\right]}}

\newcommand{\pr}[1]{\ensuremath{\mathbb{P}\left[ #1\right]}}
\newcommand{\trb}[1]{\gtr\left( #1\right)}

\DeclareMathOperator{\cov}{\mathsf{Cov}}

\DeclareMathOperator{\var}{\sf Var}

\DeclareMathOperator{\diag}{\mathrm diag}

\DeclareMathOperator{\gtr}{tr}

\newcommand{\psd}{\mathbb{S}_+}

\newcommand{\eps}{\epsilon}
\newcommand{\normal}{\mathsf{N}}

\mathtoolsset{showonlyrefs}

\title{Linear Operator Approximate Message Passing (OpAMP)} 
\author[1]{Riccardo Rossetti}
\author[2]{Bobak Nazer}
\author[1,3]{Galen Reeves}
\affil[1]{\small Department of Statistical Science, Duke University}
\affil[2]{\small Department of Electrical and Computer Engineering, Boston University}
\affil[3]{\small Department of Electrical and Computer Engineering, Duke University}
\date{}

\begin{document}
\maketitle

\abstract{This paper introduces a framework for approximate message passing (AMP) in dynamic settings where the data at each iteration is passed through a linear operator. This framework is motivated in part by applications in large-scale, distributed computing where only a subset of the data is available at each iteration. An autoregressive memory term is used to mitigate information loss across iterations and a specialized algorithm, called projection AMP, is designed for the case where each linear operator is an orthogonal projection. Precise theoretical guarantees are provided for a class of Gaussian matrices and non-separable denoising functions. Specifically, it is shown that the iterates can be well-approximated in the high-dimensional limit by a Gaussian process whose second-order statistics are defined recursively via state evolution. These results are applied to the problem of estimating a rank-one spike corrupted by additive Gaussian noise using partial row updates, and the theory is validated by numerical simulations.}

\section{Introduction}
Approximate message passing (AMP) refers to a family of iterative algorithms that has been applied to high-dimensional inference problems including regression,  matrix estimation, and channel coding; for a comprehensive reference, see the recent tutorial~\cite{feng:2022}. The basic form of an AMP algorithm can be summarized as a recursion on $n$-dimensional iterates given by
\begin{align}
    x_t = M f_t(x_{<t}) - \sum_{s < t} b_{ts} f_s(x_{<s})  , \quad t = 0, 1, 2, \dots \label{eq:basicAMP}
\end{align} 
where $M$ is an $n \times n$ ``data matrix'',  the  $f_t: \bbR^{n \times t} \rightarrow \bbR^n$ are deterministic functions (with initialization $f_0 = f_0(\emptyset) \in \bbR^n$),  the $b_{ts} \in \bbR$ are scalar ``debiasing'' coefficients, and the  notation $x_{<t} = (x_0, \dotsc, x_{t-1})$ represents the collection of iterates up to time $t-1$. See Figure~\ref{fig:basicAMPblockdiagram} for an illustration. Depending on the application, the matrix $M$ is obtained from the observed data via elementary preprocessing steps such as centering and symmetrization. 

\begin{figure}[!ht]
\centering

\begin{tikzpicture}[scale=0.9,every node/.style={scale=0.9}]
    \node[invisiblenode] (beginning) at (0,0) {};
    \node[invisiblenode] (previousx) at (1,0) {};
    \node[block, label = Denoiser] (denoiser) at (2,0) {\large{$f_t$}};
    \node[invisiblenode] (belowdenoiser) at (2,-1.5) {};
    \node[datamatrix, label = Data Matrix] (M) at (6,0) {\large{$M$}};
    \node[add] (add1) at (10,0) {$\boldsymbol{+}$};
    \node[invisiblenode] (correctionterm) at (10,1) {};
    \node[label=above:Correction Term] at (10,1.3) {\large{$\displaystyle - \sum_{s < t} b_{ts} f_s(x_{<s})$}};
    \node[invisiblenode] (ending) at (11.5,0) {};

    \node[longblock] (delay) at (7.8,-1.5) {Delay};
    \node[block] (memory) at (4,-1.5) {Memory};
    \path [basicline] (delay) edge node[above]{\large{$x_{t-1}$}} (memory);
    \path [-, thick] (memory) edge node[above]{\large{$x_{<t}$}} (belowdenoiser);
    \path [basicline] (belowdenoiser) edge node[above]{} (denoiser);
    \path [basicline] (denoiser) edge node[above]{\large{$f_t(x_{<t})$}} (M);
    \path [basicline] (M) edge node[above]{\large{$M f_t(x_{<t})$}} (add1);
    \path [basicline] (correctionterm) edge node[above]{} (add1);
    \path [basicline] (add1) edge node[above]{\large{$x_t$}} (ending);
    \draw [->, thick] (10.9,0) -- (10.9,-1.5) -- (delay);

\end{tikzpicture}
\caption{Block diagram for the basic approximate message passing (AMP) recursion with full memory. The notation $x_{<t} = (x_0,\ldots,x_{t-1})$ compactly refers to the collection of iterates up to times $t-1$.}
\label{fig:basicAMPblockdiagram}
\end{figure}
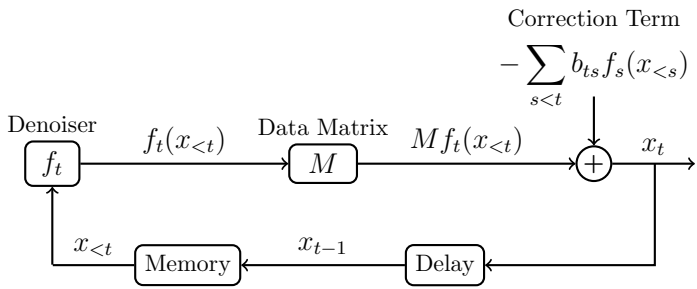

One of the key features of the AMP framework is that the behavior can be tracked precisely in high-dimension settings provided that the data matrix satisfies certain distributional assumptions. In these settings, the coefficients $b_{ts}$ can be specified in a way that both accelerates the overall convergence and guarantees that the process $\{x_t\}$ can be closely approximated by a Gaussian process $\{y_t\}$ whose mean and covariance can be efficiently computed via a recursion called state evolution (SE).  
Over the last decade, the combination of AMP-based analysis with information-theoretic arguments has provided powerful new insights into the statistical-computational tradeoffs for wide variety of problems. 

An important consideration for very large matrix operations is that the underlying computations are usually distributed across multiple servers. For instance, consider a scenario where each server is given a subset of the rows of the data matrix, and tasked to return the inner products of its rows with the current estimate. In practice, the server response times can have a long tail~\cite{Dean:2013}, and waiting for the ``stragglers'' can significantly delay the next iteration. This delay can be mitigated via replication or coded computation~\cite{Lee:2017,Dutta:2016,Yu:2017,Dutta:2020,Li:2020} to ensure that the full matrix multiplication is available once enough servers respond, at the cost of additional computation per iteration to maintain a suitable erasure-correcting code.

However, for high-dimensional inference tasks, this level of redundancy may be too conservative, and it may be more efficient to proceed to the next iterate. For example, a simple approach is to update the coordinates corresponding to responding servers, while leaving those corresponding to straggling servers unchanged. Intuitively, after applying a well-chosen denoising function, the quality of the estimate should improve, despite the missing computations. One of the main contributions of this paper is an AMP framework that naturally captures this scenario, enabling a precise performance characterization via the associated state evolution. In fact, we provide theoretical guarantees for a much broader class of problems, where the matrix multiplication at each iteration is the result of passing the data matrix through a time-varying linear operator.

From another perspective, the goal is for the iterates to converge \textit{efficiently}, e.g.,  using the fewest possible multiplications by the high-dimensional data matrix. An interesting by-product of our framework is that we can analyze the performance of AMP for carefully-scheduled computations. For instance, consider the following ``round-robin'' AMP algorithm: the rows of the data matrix are partitioned into $J$ equally-sized subsets, and the $t$-th iteration only updates the coordinates corresponding to subset $t \bmod J$. As we demonstrate below in Sections~\ref{sec:casestudy} and~\ref{sec:numerics}, for spiked matrix estimation, this update schedule can be more efficient than multiplying by the full matrix at each iteration.

\subsection{Overview of Main Results}
\paragraph*{Linear Operator AMP.}
This paper extends the scope of AMP to settings where the matrix may change with each iteration. This is modeled by  passing a fixed data matrix $M$ through a linear operator $\cL_t \colon \bbR^{n \times n} \to \bbR^{n \times n}$. We introduce Linear Operator AMP (OpAMP), which is defined by the recursion
\begin{align}
    x_t = \cL_t(M) f_t(x_{<t}) 
    - \sum_{s < t} B_{ts} f_s(x_{<s}) .  \label{eq:linAMP}
\end{align} 
Under a Gaussian assumption on the data matrix, we show how the matrices $B_{ts} \in \bbR^{n \times n}$ can be specified as a function of the $\cL_t$ to enforce approximate Gaussianity of the iterates.

\begin{figure}[!h]
\centering

\begin{tikzpicture}[scale=0.9,every node/.style={scale=0.9}]
    \node[invisiblenode] (beginning) at (0,0) {};
    \node[invisiblenode] (previousx) at (1,0) {};
    \node[block, label = Denoiser] (denoiser) at (2,0) {\large{$f_t$}};
    \node[invisiblenode] (belowdenoiser) at (2,-1.5) {};
    \node[datamatrix] (LtM) at (6,0) {\Large{$\mathcal{L}_t(M)$}};
    \node[block, label = left:Linear Operator] (linearoperator) at (6,1.35) {\large{$\mathcal{L}_t$}};
    \node[invisiblenode, label = left:Data Matrix] (M) at (6,2.5) {\large{$M$}};
    \node[add] (add1) at (10,0) {$\boldsymbol{+}$};
    \node[invisiblenode] (correctionterm) at (10,1.2) {};
    \node[label=above:Correction Term] at (10,1.5) {\large{$\displaystyle - \sum_{s < t} B_{ts} f_s(x_{<s})$}};    \node[invisiblenode] (ending) at (11.5,0) {};

    \node[longblock] (delay) at (7.8,-1.5) {Delay};
    \node[block] (memory) at (4,-1.5) {Memory};
    \path [basicline] (delay) edge node[above]{\large{$x_{t-1}$}} (memory);
    \path [-, thick] (memory) edge node[above]{\large{$x_{<t}$}} (belowdenoiser);
    \path [basicline] (belowdenoiser) edge node[above]{} (denoiser);
    \path [basicline] (denoiser) edge node[above]{\large{$f_t(x_{<t})$}} (LtM);
    \path [basicline] (LtM) edge node[above]{\large{$\mathcal{L}_t(M) f_t(x_{<t})$}} (add1);
    \path [basicline] (correctionterm) edge node[above]{} (add1);
    \path [basicline] (add1) edge node[above]{\large{$x_t$}} (ending);
    \draw [->, thick] (10.9,0) -- (10.9,-1.5) -- (delay);
    \path [basicline] (M) edge node[above]{} (linearoperator);
    \path [basicline] (linearoperator) edge node[above]{} (LtM);

\end{tikzpicture}
\caption{Block diagram for the linear operator approximate message passing (OpAMP) recursion with full memory. The notation $x_{<t} = (x_0,\ldots,x_{t-1})$ compactly refers to the collection of iterates up to times $t-1$.}
\label{fig:OpAMPblockdiagram}
\end{figure}
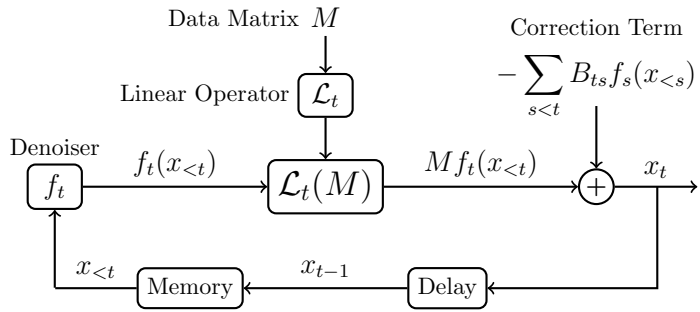

To compensate for the possibility that some of the  signal is lost at each iteration, it is essential that we allow the $f_t$ to depend on the entire history of the recursion, not just the previous update. To provide a tractable model for long-term memory we also introduce an autoregressive version of \eqref{eq:linAMP}  that has a linear dependence on the previous iterates:
\begin{align} 
    x_t = \cL_t(M) f_t(x_{t\!-\!1}) 
    +\sum_{s < t} A_{ts}  x_{s-1}
    - \sum_{s < t} B_{ts} f_s(  x_{s-1})
\end{align} 
where the matrices $A_{ts}$ can be designed such that $x_t$ provides a suitable summary of the previous iterations. 

Specializing to the case where each linear operator is given by $\cL_t(M) = \Pi_t M$ for an $n \times n$ projection matrix $\Pi_t$, we choose the matrices $A_{ts}$ to be
\begin{align}
    A_{ts} = \begin{cases}
        \Pi_t^\perp & \text{if $s = t-1$}, \\
        0 & \text{otherwise}
    \end{cases}
\end{align}
and define the projection AMP recursion 
\begin{align} \label{eq:projAMPintro}
    x_t = \Pi_t  \Big( M f_t(x_{t-1}) - \sum_{s < t} b_{ts} f_s( x_{s-1})  \Big)  + \Pi_t^{\perp}  x_{t-1} 
\end{align}
where the dependence on prior iterations is specified by the projector sequence. For the special case of commuting orthogonal projection matrices, the SE has a particularly simple recursive structure that is analysed in Section~\ref{sec:asymptotic_SE}.
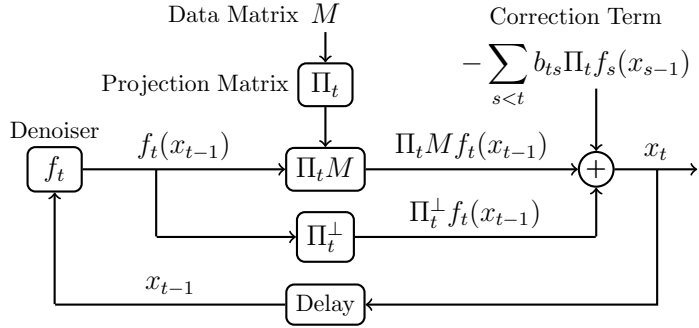
\begin{figure}
\centering

\begin{tikzpicture}[scale=0.9,every node/.style={scale=0.9}]
    \node[invisiblenode] (beginning) at (0,0) {};
    \node[invisiblenode] (previousx) at (1,0) {};
    \node[block, label = Denoiser] (denoiser) at (2,0) {\large{$f_t$}};
    \node[invisiblenode] (belowdenoiser) at (2,-2) {};
    \node[datamatrix] (PitM) at (6,0) {\large{$\Pi_t M$}};
    \node[block, label = left:Projection Matrix, inner sep = 4pt] (projection) at (6,1.25) {\large{$\Pi_t$}};
    \node[block, inner sep = 4pt] (compprojection) at (6,-1) {\large{$\Pi_t^\perp$}};
    \node[invisiblenode] (ft) at (3.5,0) {};
    \node[invisiblenode] (leftcomp) at (3.5,-1) {};
    \node[invisiblenode, label = left:Data Matrix] (M) at (6,2.3) {\large{$M$}};
    \node[add] (add1) at (10,0) {$\boldsymbol{+}$};
    \node[invisiblenode] (belowadd1) at (10,-1) {};
    \node[label=above:Correction Term] at (9.7,1.4) {\large{$\displaystyle - \sum_{s < t} b_{ts} \Pi_t f_s(x_{s-1})$}};
    \node[invisiblenode] (correctionterm) at (10,1.2) {};
    \node[invisiblenode] (ending) at (11.5,0) {};

    \node[longblock] (delay) at (6,-2) {Delay};
    \path [-, thick] (delay) edge node[above]{\large{$x_{t-1}$}} (belowdenoiser);
    \path [basicline] (belowdenoiser) edge node[above]{} (denoiser);
    \path [basicline] (denoiser) edge node[above]{\large{$f_t(x_{t-1})$}} (PitM);
    \path [basicline] (PitM) edge node[above]{\large{$\Pi_t M f_t(x_{t-1})$}} (add1);
    \path [basicline] (correctionterm) edge node[above]{} (add1);
    \path [basicline] (add1) edge node[above]{\large{$x_t$}} (ending);
    \draw [->, thick] (10.9,0) -- (10.9,-2) -- (delay);
    \path [basicline] (M) edge node[above]{} (projection);
    \path [basicline] (projection) edge node[above]{} (PitM);
    \draw [-, thick] (ft) -- (leftcomp);
    \path [basicline] (leftcomp) edge node[above]{} (compprojection);
    \path [-, thick] (compprojection) edge node[above]{\large{$\Pi_t^\perp f_t(x_{t-1})$}} (belowadd1);
    \path [basicline] (belowadd1) edge node[above]{} (add1);

\end{tikzpicture}
\caption{Block diagram for the projection approximate message passing (AMP) recursion with one-step memory.}
\label{fig:projectionAMPblockdiagram}
\end{figure}

\paragraph*{Stochastic updates via Projection AMP.} 
Consider a scenario where the rows of the data matrix $M$ are split up among multiple servers, which are tasked to return the inner products of their rows with the current estimate $f_t(x_{t-1})$. As motivated above, only a subset of the servers respond by the deadline. Alternatively, consider a scenario where only a subset of the rows are scheduled for each iteration, to reduce the computational load. Both of these scenarios can be modeled using \eqref{eq:projAMPintro} where each $\Pi_t$ is a diagonal $0$-$1$ matrix indicating which subset of the rows are  updated at each iteration and the orthogonal complement $\Pi_t^\perp = \Id - \Pi_t$ is used to retain signal memory via the term $\Pi_t^{\perp} x_{t-1}$.

Our formulation and results enables us to precisely characterize the effective signal-to-noise ratio at each iteration. As discussed below in Section~\ref{sec:relatedwork}, various special cases of our framework have been previously studied in the literature, but, to the best of the authors' understanding, the general framework for linear operators has not been considered directly. Based on the results of our simulations (Section~\ref{sec:numerics}), we suggest that our projection method is potentially useful in accelerating convergence of standard AMP recursions, even in non-distributed settings, under certain models for the computational cost per iteration.

\paragraph*{Asymptotic Gaussianity.} Beyond the algorithmic developments, a further contribution of this paper is the framework we use to describe the Gaussian approximation to the AMP iterates. Following the seminal work in \cite{Donoho:2009,Bayati:2011,javanmard:2013}, most of the literature on AMP restricts to settings where the functions $f_t$ are separable (i.e., the $i$-th output depends only on the $i$-th input) and then provides  convergence guarantees with respect to the  empirical measure of the iterates over a finite number of iterations; see \cite{feng:2022} for more details. 

In this paper, we  build on the general framework for non-separable functions \cite{berthier:2020,gerbelot:2023,Montanari:2022}, where (pseudo)-Lipschitz continuity is the only assumption placed on the function sequence and convergence is assessed in terms a of sequence of suitably normalized pseudo-Lipschitz test functions. This mode of convergence does not imply convergence in distribution for the underlying iterates because the space on which these test functions are applied is scaling with the problem size. Theorem~\ref{thm:XnYnGausslim} in the Appendix provides a dual interpretation for this mode of convergence, that is based on the existence of a coupling between the AMP iterates $\{x_t\}$ and a Gaussian approximation $\{y_t\}$. Our main theorems establish conditions under which the normalized difference over a fixed number of iterates $T$ convergence to zero in probability: 
\begin{align}
    \frac{\|x_{\scriptscriptstyle{\le T}} - y_{\scriptscriptstyle{\le T}} \|}{\sqrt{n}} \xrightarrow[n \to \infty]{\mathrm{p}} 0
\end{align}
where $x_{\le T} $ denotes the concatenation of $x_0, \dots x_T$ and $\|\cdot\|$ is the Euclidean norm. To the best of our understanding, this type of coupling argument has not appeared previously in the AMP literature. 
 
\subsection{Case Study: Power Iteration with Partial Updates} \label{sec:casestudy} 
As an illustrative example of application of our results, we consider the problem of constructing an estimate for a random point $\theta \sim \mathsf{Unif}(\{\pm 1\}^n\})$ given matrix observation $M$ from a spiked Wigner model
\begin{align} \label{eq:spikedMatrixUnitSignal}
    M = \frac{\lambda}{n} \theta\theta^\top + Z
\end{align}
where $\lambda > 0$ and $Z$ is an $n \times n$ random matrix drawn independently of $\theta$ from the Gaussian orthogonal ensemble $\mathsf{GOE}(n)$, i.e., $Z$ is symmetric with independent $\normal(0, 1/n)$ entries above the diagonal and independent $\normal( 0, 2/n)$ entries on the diagonal. 
A natural approach is to construct an estimator of $\theta$ based on properties of the spectrum of  $M \in \bbR^{n \times n}$,
\begin{align}
    M = \sum_{i=1}^n \xi_i v_i v_i^\top,
\end{align}
for eigenvalues $|\xi_1| \ge |\xi_2| \ge \dots \ge |\xi_n| \in \bbR$ sorted in order of magnitude. From standard results in random matrix theory \cite{Benaych:2011}, the value of $\lambda$ determines whether the angle $|\langle \theta, v_1\rangle|$ is non-vanishing in the large-$n$ limit. In particular,
\begin{align}
    |\langle \theta, v_1 \rangle|^2 \xrightarrow[n \rightarrow \infty]{\mathrm{a.s.}} \begin{cases}
        1 - \frac{1}{\lambda^2} & \text{if $\lambda > 1$} \\ 0 & \text{otherwise}
    \end{cases} \qquad \xi_1 \xrightarrow[n \rightarrow \infty]{\mathrm{a.s.}} \begin{cases}
        \lambda + \frac{1}{\lambda} & \text{if $\lambda > 1$} \\ 2 & \text{otherwise}
    \end{cases}
\end{align}
Conversely, when $\lambda < 1$, it is known that no procedure is capable of providing an estimate of $\theta$ based on $M$ that achieves non-vanishing correlation with the ground truth \cite{perry:2018}. Thus, a reasonable approach to estimate $\theta$ is to use the leading eigenvector of $M$, $v_1$, as our estimator.

A classical technique to recover $v_1$ is the power method, or power iteration~\cite{golub:2013aa}. For some initial value $x_0 \in \bbR^n$, the power iteration is given by the recursion
\begin{align}
   \hat{\theta}_t = \frac{\sqrt{n}}{\|x_t\|} x_t, \qquad  x_{t+1} = M \hat{\theta}_t   \label{eq:basicpower}
\end{align}
Let $\alpha \coloneqq \cos^{-1}(|\langle \frac{\hat{\theta}_0}{\sqrt{n}}, v_1 \rangle| )$ be the angle between the initialization $x_0$ and the leading eigenvalue $v_1$. Classical convergence bounds~\cite[Theorem 8.2.1]{golub:2013aa} guarantee that 
\begin{align} \label{eq:PowerIter}
    1 -  |\langle v_1, \hat{\theta}_t \rangle |^2 \leq \tan^2(\alpha) \Big| \frac{\xi_2}{\xi_1} \Big|^{2t} \ . 
\end{align}
Hence, under the mild condition $\alpha > 0$ (i.e., $x_0$ does not lie in the nullspace of $v_1$), the estimates $\hat{\theta}_t$ converge to $v_1$ geometrically fast as long as the spectral gap is greater than one. This result is general and holds for any symmetric matrix $M$.

Interpreting the projection onto the $\sqrt{n}$-sphere $x \mapsto \sqrt{n} x / \|x\|$ to be a denoising function, a standard AMP correction term can be added to the power method to obtain the ``AMP-adjusted" recursion
\begin{align} \label{eq:ampPowerIter}
    \hat\theta_{t} = \frac{\sqrt{n}}{\|x_t\|} x_t , \qquad x_{t+1} = M \hat\theta_{t} - \frac{1}{\sqrt{n}\|x_{t}\|} \hat{\theta}_{t-1}
\end{align}
with the convention that $\hat\theta_{-1} \equiv 0$. When $M$ is sampled from a spiked matrix model, this iteration can be tracked accurately in the large-$n$ limit via SE. Furthermore, when $\lambda > 1$, the iterates $\{x_t\}$ can be shown to converge to the fixed point $\sqrt{n} \lambda v_1$, a multiple of the leading eigenvector. 

For applications with very large data matrices, it may only be possible to operate on a subset of the full data matrix $M$ at each iteration. Specifically, consider the setting where only a fraction of the rows of $M$ are available at each iteration $t$, i.e. only $\Pi_t M$ is observed where $\Pi_t = \diag(\delta_t)$ and $\delta_t \in \{ 0,1\}^n$ is a binary vector indicating which rows are updated in the $t$-th iteration.

for some diagonal 0--1 matrices $\{\Pi_t\}$. As discussed above, this may be due to straggler nodes in a distributed setting or deliberate scheduling. Based on the results in this paper, we propose the following AMP-corrected power method with partial updates:
\begin{align} \label{eq:opampPowerIter}
    \hat\theta_t = \frac{\sqrt{n}}{\|x_t\|}, \quad x_{t+1} = \Pi_t\left( M\hat\theta_t - \frac{1}{\sqrt{n}\|x_t\|} \sum_{s<t} \trb{\Pi_{t-1}^\perp\dotsm \Pi_{s+1}^\perp\Pi_s} \hat\theta_{s} \right) + \Pi_t^\perp x_t
\end{align} where the complementary projections are simply $\Pi_t^\perp = \Id - \Pi_t$. This recursion is practical as it does not require knowledge of any of the model parameters and its computational cost per iteration $t$ is dominated by the matrix multiply $\Pi_tMx_t$, which involves $O(nk_t)$ operations, where $k_t = \operatorname{rank}(\Pi_t)$ (or the number of nonzero diagonal elements).

Despite the apparent complexity, the trace terms $\trb{\Pi_{t-1}^\perp\dotsm \Pi_{s+1}^\perp\Pi_s}$ can be computed efficiently in a recursive fashion in $O(nt)$ operations for time step $t$. 
This can be achieved via an equivalent recursion that uses the vectors $\gamma_s^{(t)} \in \{0,1\}^n, 0\leq s \leq t$ to indicate the coordinates that, as of time $t$, were most recently updated at time $s$. Specifically, the $t$-th iterate is formed by
\label{eq:opampEquiv}
\begin{align}
    &\gamma_{tt} = \Pi_t \mathsf{1} \nonumber \\
    &\gamma_{ts} = \Pi_t^\perp \gamma_{t-1,s},\  0 \leq s < t \nonumber \\
    &\hat\theta_t = \frac{\sqrt{n}}{\|x_t\|}x_t, \quad x_{t+1} = \Pi_t \left( M \hat\theta_t  - \frac{1}{\sqrt{n}\|x_t\|} \sum_{s<t} \| \gamma_{t-1,s}\|  \hat\theta_s \right)   + \Pi_t^\perp x_t 
\end{align} 
where $\mathsf{1}$ denotes the all-ones vector. Alternatively, see the implementation in Algorithm~\ref{alg:opamp}.

In Section~\ref{sec:powerIterSE}, we provide single-letter formulas for the SE dynamics of \eqref{eq:opampPowerIter} under some structural assumptions on the signal $\theta$ and update schedule $\{\Pi_t\}$. Note that the AMP correction term is non-standard due to its long-term memory. Intuitively, this is due to coordinates that have not been updated for more than one time step. Formally, it follows from the correction term in our OpAMP framework.

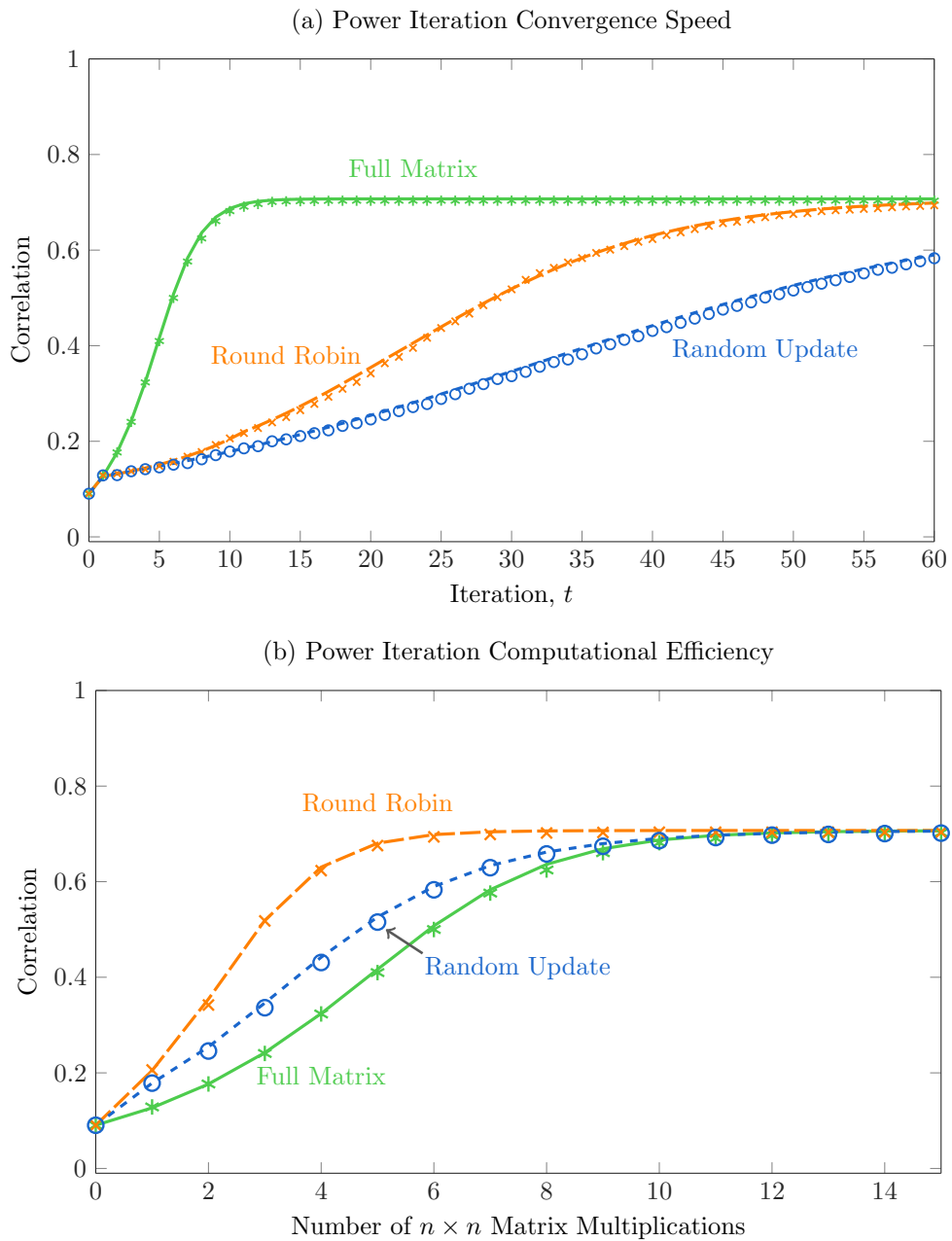
\begin{figure}
\begin{center}
\subfigure{
\label{fig:pi_intro_iter_plot}
\def\datafilename{"fig_data/banff-gaus.csv"} 

\begin{tikzpicture}
\begin{axis}[%
width=0.7\textwidth,
height=0.4\textwidth,
scale only axis,
separate axis lines,
every outer x axis line/.append style={white!15!black},
every x tick label/.append style={font=\color{white!15!black}},
xmin=1,
xmax=100,
xlabel={Iteration, $t$},
every outer y axis line/.append style={white!15!black},
every y tick label/.append style={font=\color{white!15!black}},
ymin=-0.01,
ymax=1,
ylabel={Correlation},
ylabel near ticks,
title = {(a) Power Iteration Convergence Speed},
legend style={at={(0.98,0.96)},anchor=south east,draw=white!15!black,fill=white,legend cell align=left}
]

\addplot [color=plot-green,solid, line width=1.4pt] table[col sep=comma, x="iter", y="se_fullmat"] {\datafilename};

\addplot [datamarkstyle, color = plot-green, mark = asterisk]
            table [col sep = comma, x="iter",y="amp_fullmat"] {\datafilename} node[pos=0.15,above = .1cm ] {Full Matrix};

\addplot [color=orange,dash pattern=on 8pt off 2pt, line width=1.4pt] table[col sep=comma, x="iter", y="se_robin"] {\datafilename};

\addplot [datamarkstyle, color = orange, mark = x]
            table [col sep = comma, x="iter",y="opamp_robin"] {\datafilename} node[pos=0.15, above = .9cm ] {Round Robin};

\addplot [color=plot-darkblue,dashed,line width=1.4pt] table[col sep=comma, x="iter", y="se_rand"] {\datafilename}node[pos=0.36, below = .55cm ] {Random Update};

\addplot [datamarkstyle, color = plot-darkblue, mark = +]
            table [col sep = comma, x="iter",y="opamp_rand", y error="amp_rand_sd"] {\datafilename};
\end{axis}
\end{tikzpicture}%
}
\subfigure{
\label{fig:pi_intro_comp_plot}
\def\datafilename{"fig_data/banff-norm-no-downsample.csv"} 

\begin{tikzpicture}
\begin{axis}[%
width=0.7\textwidth,
height=0.4\textwidth,
scale only axis,
separate axis lines,
every outer x axis line/.append style={white!15!black},
every x tick label/.append style={font=\color{white!15!black}},
xmin=1,
xmax=12,
xlabel={Number of $n \times n$ Matrix Multiplications},
every outer y axis line/.append style={white!15!black},
every y tick label/.append style={font=\color{white!15!black}},
ymin=-0.01,
ymax=1,
ylabel={Correlation},
ylabel near ticks,
title = {(b) Power Iteration Computational Efficiency},
legend style={at={(0.98,0.96)},anchor=south east,draw=white!15!black,fill=white,legend cell align=left}
]

\addplot [color=orange,dash pattern=on 8pt off 2pt, line width=1.4pt] table[col sep=comma, x="iter_norm", y="se_robin"] {\datafilename};

\addplot [datamarkstyleB, color = orange, mark = x]
            table [col sep = comma, x="iter_norm",y="opamp_robin"] {\datafilename} node[pos=0.4, above = .1cm ] {Round Robin};

\addplot [color=plot-darkblue,dashed,line width=1.4pt] table[col sep=comma, x="iter_norm", y="se_rand"] {\datafilename}node[pos=0.32, above = -0.05cm, rotate=28] {Random Update};


\addplot [datamarkstyleB, color = plot-darkblue, mark = +]
            table [col sep = comma, x="iter_norm",y="opamp_rand", y error="amp_rand_sd"] {\datafilename};

\addplot [color=plot-green,solid, line width=1.4pt] table[col sep=comma, x="iter", y="se_full"] {\datafilename};

\addplot [datamarkstyleB, color = plot-green, mark = asterisk]
            table [col sep = comma, x="iter",y="amp_full"] {\datafilename} node[pos=0.032,below = .1cm, rotate=28 ] {Full Matrix};
            
\end{axis}
\end{tikzpicture}%
}
\end{center}
\caption{Correlation $\frac{1}{n} \langle \theta, \hat{\theta}_t \rangle
$ attained after $t$ iterations for the full-matrix, round-robin, and random-update protocols. The denoising function is simply projection onto the $\sqrt{n}$-sphere $f_t(x) = \sqrt{n}/\|x\|\, x$, meaning that, modulo the correction term, these protocols are variations on power iteration. The curves represent the theoretical predictions from the state evolution from Theorem~\ref{thm:projAMP} and the marks represent the empirical performance averaged over $100$ trials, starting with the same initialization. All three protocols converge to the same fixed point. On the top, the correlation is plotted with respect to the number of iterations, for which the full-matrix updates are intuitively faster. On the bottom, the same data is plotted with respect to the number of effective multiplications by an $n \times n$ matrix, which serves as a proxy for the computational complexity. Under this metric, the round-robin protocol is the most efficient. }
\label{fig:pi_intro_plots}
\end{figure}

\begin{algorithm} 
\caption{Power Iteration with Partial Updates} \label{alg:opamp}
\begin{algorithmic}[1]
    \State{\textbf{Inputs: }$M \in \bbR^{n \times n},\, x_0 \in \bbR^n, \, \delta_t \in \{0,1\}^n$}
    \State $r_0 \gets \|x_0\|$ 
    \State $w_0 \gets \sum_{i=1}^n \delta_{0i}$ 
    \For{$i=1,2,\dots,n$}
        \If{$\delta_{0i} = 1$}
            \State $x_{1i} = \frac{1}{r_0} \sum_{j=1}^n M_{ij} x_{0j}$
        \Else
            \State $x_{1i} = x_{0i}$
        \EndIf
    \EndFor
    
    \For{$t=1,2,\dots,T-1$} 
        \State $r_t \gets \|x_t\|$
        \State $w_t \gets \sum_{i=1}^n \delta_{ti}$
        \For{$i=1,2,\dots,n$}
        \If{$\delta_{ti} = 1$}
            \State $x_{t+1,i} \gets \frac{1}{r_t} \left( \sum_{j=1}^n M_{ij} x_{ti} - \frac{1}{n} \sum_{s=0}^{t-1} \frac{w_s}{r_s} x_{si} \right)$
            \For{$s=0,\dots,t-1$}
                \State $w_s \gets w_s - \delta_{ti}\delta_{si}$
                \State $\delta_{si} \gets 0$
            \EndFor
        \Else
            \State $x_{t+1,i} \gets x_{ti}$
        \EndIf
        \EndFor
    \EndFor
    \State \textbf{return} $(x_1, \dots, x_T)$
\end{algorithmic}
\end{algorithm}

\paragraph*{Numerical Results.} In Figure~\ref{fig:pi_intro_plots}, the projection AMP algorithm in \eqref{eq:opampPowerIter} is applied to a size $n = 5,000$ data matrix with $\lambda = \sqrt{2}$, and the results are averaged over $100$ Monte Carlo trials, each with the same ground truth $\theta \in \{ \pm 1 \}^n$ and initialization $x_0 \in \bbR^n$ with $\langle \frac{\theta}{\sqrt{n}} , \frac{x_0}{\|x_0\|} \rangle = 0.01$. A comparison of the empirical results and the theoretical prediction obtained from the asymptotic SE is provided for the following protocols: 
\begin{itemize}
    \item \emph{Full-matrix update:} The full data matrix is applied at each iteration.
    \item \emph{Round-robin update:} The rows are partitioned into $10$ equally-sized subsets. Each iteration applies the rows in a subset, and we cycle through the subsets in a fixed order over $10$ iterations.
    \item \emph{Random update:} At each iteration, each row of the data matrix is applied independently with probability $1/10$.
\end{itemize} A formal description of the update protocols is presented in Section~\ref{sec:numerics}. Intuitively, the round-robin and random-update protocols converge more slowly than full-matrix AMP with respect to the raw iteration count. However, if we instead plot the performance with respect to effective number of $n \times n$ matrix multiplications, we find that the round-robin protocol is more efficient than the full-matrix protocol. (For certain parameter settings, the random-update protocols is also more efficient, but this is not always the case.)

\subsection{Related Work} \label{sec:relatedwork}

Most of the work in the AMP literature has focused on finite-memory versions of \eqref{eq:basicAMP} where each $f_t$ depends only on a fixed number of previous iterates\cite{Donoho:2009,Bayati:2011,rangan:2011,javanmard:2013,rush:2017aa,lesieur:2017aa, pandit:2019aa,Montanari:2021a,behne2022fundamental}. The full-memory formulation  in \eqref{eq:basicAMP} has appeared in recent work as a model for generalized first-order methods~\cite{Celentano:2020,Montanari:2022a}. Our analysis builds on the theoretical framework for non-separable functions and IID Gaussian matrices introduced by Berthier~et al. \cite{berthier:2020}, and further developed in \cite{gerbelot:2023,Montanari:2022}.

A version of the projection AMP framework of this paper was studied by Çakmak et al.~\cite{Cakmak:2022} using non-rigorous dynamical functional theory. Projection AMP is also related to recent work that uses a full memory AMP recursion to approximate the discrete-time dynamics of gradient descent and other first-order optimizations techniques~\cite{Celentano:2021, Gerbelot:2023a}. These works focus on the behavior of existing optimization techniques in scaling regimes where the number of rows updated at each iteration grows sublinearly with the problem dimension. By contrast, the main focus of this paper is to design algorithms that overcome the limitations of the dynamic data as expressed in \eqref{eq:linAMP}, e.g., by optimizing the long-term memory as function of the linear operators. 

In a slightly different direction, the special case of \eqref{eq:linAMP} where the linear operator is fixed for all iterations has been studied by a subset of the authors in the context of the matrix tensor product model \cite{reeves2020information, rossetti2023approximate}.

Beyond the setting of IID Gaussian matrices, AMP algorithms have also been proposed and analyzed for orthogonally-invariant random matrix ensembles \cite{rangan:2019,Opper:2016,ma:2017ab, 
Takeuchi:2021,Liu:2022,Fan:2022,Zhong:2021,Barbier:2023} and semirandom ensembles \cite{Dudeja:2023}. These results impose a separability assumption on the functions $\{f_t\}$, which precludes the general linear transformations used in the proof of our OpAMP framework. Extending the results in this  paper to other matrix ensembles is an interesting direction for future research. 

We also note that there has been  progress on finite sample analysis of AMP algorithms via concentration bounds \cite{Rush:2018,Cademartori:2023}, local convexity properties \cite{Celentano:2023}, and distributional decompositions \cite{Lu:2021, Li:2023}. However, all of these results are restricted to the setting of separable functions.  
Recent efforts have developed distributed, accelerated, and robust variations on the power method as well as long-run convergence guarantees under suitable conditions~\cite{Hardt:2014,Lei:2016,Xu:2018, Raja:2022, Li:2021, Xu:2022}. By choosing $f_t$ in~\eqref{eq:projAMPintro} to project onto the unit sphere, our framework includes, as a special case, a distributed power method with partial updates as well as the associated per-iteration guarantees from the state evolution. From one perspective, our correction term acts like an acceleration/momentum term.  Additionally, recent work on subspace tracking algorithms with missing data~\cite{Balzano:2018} has characterized the high-dimensional performance limit via differential equations~\cite{Wang:2018}, and it would be interesting to explore possible connections to AMP.

\section{Linear Operator AMP}
This section describes the linear operator AMP framework and states our main theoretical results. To keep the presentation streamlined, we focus on the centered version of the recursion given by
\begin{align}
    x_t = \cL_t( Z)  f_t(x_{<t}) - \sum_{s < t} B_{ts} f_s(x_{<s}) \label{eq:opAMP}
\end{align}
where $Z$ is an $n \times n$ random matrix drawn from the Gaussian orthogonal ensemble $\mathsf{GOE}(n)$, i.e., $Z$ is symmetric with independent $\normal(0, 1/n)$ entries above the diagonal and independent $\normal( 0, 2/n)$ entries on the diagonal. The extension to settings where the matrix has a low-rank signal component follows from standard arguments in the AMP literature.

These arguments are outlined in Appendix~\ref{sec:AMPcentering} for the spiked matrix model considered in Section~\ref{sec:stochastic_updates} and full details for other problem settings (e.g., regression) can be found in~\cite{feng:2022}.

Each linear operator $\cL_t$ has a (non-unique) decomposition of the form
\begin{align}
\cL_t(Z) = \sum_{k=1}^K L_{tk} Z R_{tk} \label{eq:cLt}
\end{align}
for $n \times n$ matrices $\{ L_{tk}, R_{tk} : k = 1, \dots, K\}$. We require that the both operator norm and the rank (i.e., the smallest $K$ such that \eqref{eq:cLt} holds) are bounded uniformly w.r.t.\ the problem dimension $n$. 

Our theoretical results  provide a connection between the distribution of the AMP iterates and a zero-mean Gaussian process $\{y_t\}$ whose second-order statistics are described in terms of $\{\cL_t\}$ and $\{f_t\}$ via a recursive process known as \emph{state evolution}. Starting with $\cov(y_0) = \frac{1}{n} \|f_0\|^2 \Id $, the covariance at time $t$ is defined by the covariance up to $t -1$ according to
\begin{align}
\label{eq:SEopAMP}
\cov(y_s, y_t)  
& = \sum_{l,k=1}^K  q_{sltk}  L_{sl} L_{tk}^\top, \nonumber \\
   q_{sltk} & \coloneqq \frac{1}{n} \ex{\langle R_{sl} f_{s}(y_{<s} ), R_{tk} f_{t}(y_{<t}) \rangle } 
\end{align}
This construction holds for any collections of matrices  $\{ L_{tk}, R_{tk} \}$ satisfying the decomposition in \eqref{eq:cLt}.

\begin{assumption}\label{ass:ft}
Each $f_t \colon \bbR^{n \times t} \to \bbR^n$ is $L$-Lipschitz continuous and satisfies $\frac{1}{\sqrt{n}}\|f_t(0)\| \le C$ where $C,L$ are positive constants that do not depend on $n$.  
\end{assumption}

\begin{assumption}\label{ass:Lt}
Each $\cL_t$ can be decomposed in the form given in \eqref{eq:cLt}  with $\|R_{tk}\|_{op}, \|L_{tk}\|_{op} \le C'$ for all $t \in \bbN_0, k=1,\dots,K$,  where $C'$ is a positive constant that does not depend on $n$.
\end{assumption}

\begin{theorem}\label{thm:opAMP}
Let $\{x_t\}$ be generated by \eqref{eq:opAMP} and let $\{y_t\}$ be the zero-mean Gaussian process defined by the SE \eqref{eq:SEopAMP}. Suppose Assumptions~\ref{ass:ft} and \ref{ass:Lt} hold,  $Z \sim \mathsf{GOE}(n)$, and 
the matrices $\{B_{ts} : 0 \le s < t \}$ are given by
\begin{align}
B_{ts}  
   & =  \sum_{k,l=1}^K \frac{1}{n} \gtr\left( R_{tk} \ex{\mathsf{D}_{s} f_t\left( y_{<t}  \right)} L_{sl}   \right)\, L_{tk}   R_{sl} \label{eq:BtsopAMP}
\end{align}
Here, the notation $\mathsf{D}_s$ indicates the Jacobian matrix of $f_t(x_{<t})$ computed w.r.t. the input vector $x_s$. Then, for any fixed number of iterations $T$, there exists  a sequence (in $n$) of couplings between $x_{\scriptscriptstyle{\le T}}$ and $y_{\scriptscriptstyle{\le T}}$ such that
\begin{align}
    \frac{\|x_{\scriptscriptstyle{\le T}} - y_{\scriptscriptstyle{\le T}} \|}{\sqrt{n}} \xrightarrow[n \to \infty]{\mathrm{p}} 0 
\end{align}
\end{theorem}

\begin{remark}\label{remark:bts_hat}
The debiasing matrices  $B_{ts}$ given in Theorem~\ref{thm:opAMP} are deterministic matrices specified by expectations with respect to the Gaussian process $\{y_t\}$.  For practical applications, it can be simpler to replace these matrices with approximations $\hat{B}_{ts}$ that are measurable with respect to the previous iterates. For example, one may use the empirical estimate 
\begin{align}
    \hat{B}_{ts} =   \sum_{k,l=1}^K \frac{1}{n} \gtr\left( R_{tk} \mathsf{D}_{s} f_t\left( x_{<t}  \right) L_{sl}   \right)\, L_{tk}   R_{sl} 
\end{align}
where the expectation with respect to $y_{<t}$ is replaced by the sample path of $x_{<t}$. So long as these approximations are consistent in the sense that $\|\hat{B}_{ts} - B_{ts}\|_{op}$ converges to zero in probability as $n \to \infty$,  it can be shown that the Gaussian approximation in Theorem~\ref{thm:opAMP} also holds for the recursion defined with $\hat{B}_{ts}$. The proof uses the same recursive argument as the proof of \cite[Corollary~2]{berthier:2020} and is not repeated here. 
\end{remark}

\subsection{Autoregressive Linear Operator AMP}

While the general formulation in \eqref{eq:opAMP} allows for arbitrary dependence on prior iterations, the question remains of how the  $f_t$ should be optimized as a function of the linear operators. Motivated by practical considerations, we introduce an autoregressive version of \eqref{eq:opAMP}, that uses  a weighted linear combination of previous updates:
\begin{align}
    x_t = \cL_t( Z)  f_t(x_{t-1}) + \sum_{s <t } A_{ts} x_s  - \sum_{s < t} B_{ts} f_s(x_{t-1}) \label{eq:opAMPlinmemory}
\end{align}
Here,  the $n \times n$ matrices $ A_{ts}$ describe  the long-term dependence and the functions $f_t$ are applied only to the prior iteration, reducing the complexity of both the implementation and the analysis. 

To describe the SE, we define the collection of $n \times n$ matrices $\{C_{st} : 0 \le s < t\}$ according to 
\begin{align}
    \begin{bmatrix}
        \Id_n & \\
        -A_{10} & \Id_n & \\
         \vdots &&  \ddots \\
         -A_{t0}&   \cdots & - A_{t,t-1}&  \Id_n 
    \end{bmatrix}^{-1}\! = \begin{bmatrix}
        \Id_n & \\
        C_{10} & \Id_n & \\
         \vdots &&  \ddots \\
         C_{t0}&   \cdots & C_{t,t-1}&  \Id_n 
    \end{bmatrix} 
\end{align}
The fact that these matrices are block unitriangular  ensures that the inverse exists. Furthermore, using the convention that $C_{tt} = \Id_n$, the  mapping from $\{A_{ts}\}$ to $\{C_{ts}\}$  can be expressed recursively via 
\begin{align}
C_{ts} = \sum_{r =0}^{t-1} A_{t r} C_{rs}  
\label{eq:AtoC}
\end{align} 
The distribution of the iterates is compared with a zero-mean Gaussian process $\{y_t\}$ whose covariance is defined recursively according to
\begin{align}
   \cov(y_s, y_t)  
& = \sum_{s' \le s, t' \le t} \,\sum_{l,k=1}^K   q_{s'lt'k} C_{ss'}  L_{s'l}( L_{t'k} C_{tt'})^\top \label{eq:SEopAMPlinmemory}
\end{align}
where $q_{sltk}$ is defined in \eqref{eq:SEopAMP}. 

\begin{assumption}\label{ass:Ats}
 $\|A_{ts}\|_{op} \le C''$ for all $s,t$  where $C''$ is a positive constant that does not depend on $n$.
\end{assumption}

\begin{theorem}\label{thm:opAMPlinmemory}
Let $\{x_t\}$ be generated by \eqref{eq:opAMPlinmemory} and let $\{y_n\}$ be the zero-mean Gaussian process defined by the SE. Suppose Assumptions~\ref{ass:ft}, \ref{ass:Lt}, and \ref{ass:Ats} hold, $Z \sim \mathsf{GOE}(n)$, and 
the matrices $\{B_{ts} : 0 \le s < t \}$ are given by
\begin{align}
B_{ts} 
& = \sum_{k,l=1}^K \frac{1}{n} \gtr\big( R_{tk} \ex{\mathsf{D} f_t\left( y_{t-1} \right) }  C_{t-1, s} L_{sl} \big)L_{tk} R_{sl}
\label{eq:BtsopAMPlinmemory}
\end{align}
where $\{ L_{tk}, R_{tk} \}$ provide the decomposition of $\cL_t$ given in \eqref{eq:cLt} and $\{C_{ts}\}$ are defined by \eqref{eq:AtoC}.
Then, for any fixed number of iterations $T$, there exists  a sequence (in $n$) of couplings between $x_{\scriptscriptstyle{\le T}}$ and $y_{\scriptscriptstyle{\le T}}$ such that
\begin{align}
    \frac{\|x_{\scriptscriptstyle{\le T}} - y_{\scriptscriptstyle{\le T}} \|}{\sqrt{n}} \xrightarrow[n \to \infty]{\mathrm{p}} 0
\end{align}
\end{theorem}

Note that in view of \eqref{eq:AtoC}, the matrices $C_{t-1,s}$ appearing in the correction term at time $t$ can be computed in recursively in terms of the matrices $C_{t-2,s}$ used in the previous iteration. 

\subsection{Projection AMP}
The projection AMP framework is a specialization of \eqref{eq:opAMPlinmemory} where each linear operator is given by $\cL_t(Z) = \Pi_t Z$ for an $n \times n$ projection matrix $\Pi_t$ and the autoregressive linear memory term is the complementary projection matrix $\Pi_t^{\perp} = \Id - \Pi_t$ applied to the past iteration. 

The projection AMP recursion can be expressed as
\begin{align}
    x_t & = \Pi_t \Big(  Z   f_t(x_{t-1})  - \sum_{s < t} b_{ts} f_s(x_{s-1})  \Big) + \Pi_{t}^\perp x_{t-1}  \label{eq:projAMP}
\end{align}
A useful property of this formulation is that the memory terms depend only on the projector sequence and the debiasing terms are described by scalars. 

The additional structure of projection matrices also leads to simplifications for the SE. The covariance of the zero-mean Gaussian process $\{y_t\}$ is given by 
    \begin{align}
   \cov(y_s, y_t)  
& = \sum_{s' \le s, t' \le t}    \frac{1}{n} \ex{\langle  f_{s'}(y_{s'-1} ),  f_{t'}(y_{t'-1}) \rangle } C_{ss'}  \Pi_{s'}( \Pi_{t'} C_{tt'})^\top \label{eq:SEprojAMP}
\end{align}
where the matrices $C_{ts}$ are defined by
\begin{align}
C_{ts}
&= \begin{dcases}
\Id_n & s = t\\
    \Pi_t^{\perp} \Pi_{t-1}^{\perp} \cdots \Pi_{s+2}^\perp \Pi_{s+1}^{\perp} & 0 \le s < t
\end{dcases} \label{eq:Ctsproj}
\end{align}

\begin{theorem}\label{thm:projAMP}
Let $\{x_t\}$ be generated by \eqref{eq:projAMP} and let $\{y_n\}$ be the zero-mean Gaussian process defined by the SE. Suppose Assumption~\ref{ass:ft} holds, $Z \sim \mathsf{GOE}(n)$, and 
the scalars $\{b_{ts} : 0 \le s < t \}$ are given by
\begin{align}
 b_{ts} & =  \frac{1}{n} \gtr\left( \ex{ \mathsf{D} f_t(y_{t-1}) } C_{t-1,s}  \Pi_{s} \right) \label{eq:bts_projAMP}
\end{align}
Then, for any fixed number of iterations $T$, there exists  a sequence (in $n$) of couplings between $x_{\scriptscriptstyle{\le T}}$ and $y_{\scriptscriptstyle{\le T}}$ such that
\begin{align}
    \frac{\|x_{\scriptscriptstyle{\le T}} - y_{\scriptscriptstyle{\le T}} \|}{\sqrt{n}} \xrightarrow[n \to \infty]{\mathrm{p}} 0
\end{align}
\end{theorem}

The SE for projection AMP admits further simplifications for the special case of commuting orthogonal projections, i.e., each $\Pi_s$ is symmetric and $\Pi_s \Pi_t = \Pi_t \Pi_s$ for all $s,t$. Starting with \eqref{eq:SEprojAMP} and then using the fact that the $\Pi_t$ and $C_{ts}$ commute, one finds that $\cov(y_t)$ satisfies the simple recursion
\begin{align}
\cov(y_t) & = 
\frac{1}{n} \ex{\|f_t(y_{t-1})\|^2} \Pi_t  +  \cov(y_{t-1}) \Pi_t^\perp  \label{eq:covprojAMP}
\end{align} 
where $ \cov(y_{t-1}) \Pi_t^\perp = \Pi_t^\perp \cov(y_{t-1})$ is symmetric. In particular, if $\Pi_0 = \Id_n$ and every $f_t$ is supported on the sphere of radius $\sigma \sqrt{n}$ then it follows that $\cov(y_t) = \sigma^2\Id_n$ is constant across iterations.   

\section{Matrix Estimation with Partial Updates}
\label{sec:stochastic_updates}
In this section, we show how our linear operator framework can be applied to settings where the entire data matrix cannot be applied at each iteration. For concreteness, we focus on the rank-one spiked matrix model
\begin{align}
    M = \frac{\lambda}{n} \theta \theta^\top + Z \label{eq:spikedMatrix}
\end{align}
where $\lambda \ge 0$ is a positive scalar,  $\theta = (\theta_1, \dots, \theta_n) \in \bbR^{n}$ is the unknown signal, and  $Z \sim \mathsf{GOE}(n)$ is additive noise. The goal is to recover $\theta$ from $M$ subject to the constraints that only a subset of the rows of $M$ can be accessed at each iteration. 

The update constraints are modeled using projection AMP with projections of the form $\Pi_t = \diag(\delta_t)$ where $\delta_t \in \{ 0,1\}^n$ is a binary vector indicating which rows can be updated in the $t$-th iteration. 
For a given sequence of ``denoising'' functions $\{f_t\}$, we construct a sequence of estimates $\{\hat{\theta}_t\}$ using the following version of projection AMP:
\begin{align} \label{eq:diagProjAMPupdates}
\hat{\theta}_t & = f_t(x_{t- 1})\nonumber\\
x_t
& = \delta_t \circ  \Big( M \hat{\theta}_t  - \sum_{s<t} b_{ts} \hat{\theta}_s \Big) + (\mathsf{1}-  \delta_t) \circ x_{t- 1} 
\end{align}
Here, $\circ$ denotes the elementwise (Hadamard) product and $\mathsf{1}$ denotes the all ones vector. The scalar debiasing coefficients $b_{ts}$ are defined as a function of the SE according to \eqref{eq:bts_tau}. To circumvent some cumbersome details that arise with a generic initialization, we will assume throughout this section   that every row of the matrix is updated in the the first time step, i.e.,  $\delta_0 \equiv \mathsf{1}$ is the all-ones vector. 

\paragraph*{State evolution.} 
Combining Theorem~\ref{thm:projAMP} with the recentering arguments given in Appendix~\ref{sec:AMPcentering}, we find that the iterates from the recursion \eqref{eq:diagProjAMPupdates} are well approximated by a Gaussian process $\{y_t\}$, whose mean and covariance are defined by a two-parameter SE.  
The SE is described in terms of the scalar overlap parameters
\begin{align} 
    q_t &= \frac{1}{n} \ex{\| f_t(y_{t-1})\|^2}; \quad 
    r_t = \frac{1}{n} \ex{\langle \theta , f_t(y_{t-1})\rangle} \label{eq:qtrt}
\end{align}
where $q_0 = \|f_0\|^2/n, \ r_0 = \langle \theta, f_0 \rangle / n$ are the overlaps arising from the initial estimate $\hat{\theta}_0 = f_0 \in \bbR^n$. 
Starting with $\ex{ y_{0} } = \lambda  r_0 \theta$ and $\cov(y_0) = q_0\Id_n$, the mean and covariance of $\{y_t\}$ are updated recursively according to 
\begin{align} \label{eq:spikedSE}
    \ex{y_t} & = \lambda r_t \delta_t \circ  \theta + ( 1 - \delta_t) \circ  \ex{ y_{t-1}} \ , \nonumber \\
    \cov( y_t) & = q_t \diag(\delta_t)  +  \cov( y_{t-1}) ( \Id_n - \diag( \delta_t)) 
\end{align}
\begin{remark}
    State evolution is defined in terms of the ground truth signal $\theta$, that can be assumed to be either fixed or random but is generally unknown. Any computationally feasible AMP algorithm will not explicitly rely on knowledge of $\theta$, and instead will estimate $r_t$ via concentration arguments based on structural information about theta (e.g., sparsity level, noise variance, limiting empirical distribution).
\end{remark}
A useful property of the SE is that, for each time step $t$, the distribution of the $i$-th component of the Gaussian vector $y_t = (y_{1t}, \dots, y_{tn})$  depends only on the signal component $\theta_i$ and the last time step in which  $i$-th row of the matrix was updated. To see this, observe that the $\cov(y_t)$ is diagonal for all iterations, and thus each $y_t$ has independent components. For a given indicator sequence $\{\delta_t\}$, the index for the most recent update of the $i$-th row before time step $t$ is encoded by the function
 $\tau : \bbN \times \{1,\dots,n\} \rightarrow \bbN_0$, defined by
\begin{align} \label{eq:tauDefinition}
 \tau(t, i) \coloneqq \max \{ s  \in  \{ 0, 1, \dots, t-1\} : \delta_{si} = 1 \}
\end{align}
where $\delta_{si}$ is the $i$-th element of the binary vector $\delta_s$.

From the recursive structure in \eqref{eq:spikedSE}, it follows that every component that was last updated at time step $s$ is described by a scalar Gaussian noise model with  parameters $(q_s,r_s)$.  Specifically, the variables  $y_{t1}, \dots, y_{tn}$  are distributed independently with 
\begin{align}
 \tau(t+1,i) = s \quad \implies  \quad    y_{ti} \sim \normal\left( \lambda r_{s} \theta_i, q_{s} \right)  \label{eq:yts_model}
\end{align}
for all $t \in \bbN_0$. We note that the function $\tau$ provides an alternative representation of the indicator sequence $\{\delta_t\}$, which can be recovered via the correspondence $\delta_{ti} = \one\{ \tau(t+1,i) = t\}$. 

\paragraph*{Debiasing coefficients.}
The function $\tau$ also provides a compact representation of the debiasing coefficients defined by \eqref{eq:bts_projAMP}. Observe that the $n \times n$ matrix $C_{t-1,s} \Pi_s$ appearing in the specification of $b_{ts}$ is given by 
\begin{align}
  C_{t-1,s} \Pi_s = \begin{dcases}
  \diag( (\mathsf{1}-\delta_{t-1})  \circ(\mathsf{1} -  \delta_{t-2})  \cdots  (\mathsf{1}- \delta_{s+1}) \circ \delta_s), &  s  =  0, 1, \dots, t-2\\
  \diag(\delta_{t-1}) & s = t-1
  \end{dcases}
    \label{eq:CtsPs}
\end{align}
This is a diagonal binary matrix with whose $i$-th diagonal entry is the indicator of the event that the $i$-th row was last updated at time step $s$, i.e., the event $\{ \tau(t,i) = s\}$. Plugging \eqref{eq:CtsPs} into the definition of projection AMP corrections \eqref{eq:bts_projAMP}, we see that the debiasing coefficients can be expressed as
\begin{align}
    b_{ts} = \frac{1}{n} \sum_{i=1}^n \one\{ \tau(t, i) = s \} \ex{ \mathsf{D}_{ii} f_t(y_{t-1}) } \label{eq:bts_tau}
\end{align}
where $\mathsf{D}_{ii}$ denotes the partial derivative of the $i$-th output $f_{ti}(y_{t-1})$ with respect to the $i$-th input $y_{t-1,i}$. 

\subsection{Asymptotic State Evolution}\label{sec:asymptotic_SE}
In this section, we obtain a simplified characterization of the SE by focusing on the high-dimensional limit for sequence of problems, with increasing dimension $n$,  where the signal $\theta \in \bbR^n$, initialization $\hat{\theta}_0 \in \bbR^n$,  and indicator sequence $\{\delta_t\} \in \{0,1\}^{n \times \bbN_0}$ together satisfy a decoupling property. 
\begin{assumption} \label{as:singLetterAssumptions}
    $M$ is given by \eqref{eq:spikedMatrix} and the Projection AMP recursion \eqref{eq:diagProjAMPupdates} satisfies the following conditions:
    \begin{enumerate}
    \item For each  $t \in \bbN_0$, the joint empirical measure of $\{ ( \theta_i, \hat{\theta}_{0i} ,  \delta_{0i}, \dots, \delta_{ti}) : i \in [n]\} $ converges in quadratic Wasserstein distance to a limiting probability measure the form  $\mu \otimes \nu_t$ where $\mu$ is a distribution on $\bbR^2$ whose marginals have unit second moments and $\nu_t$ is is the distribution of the first $t$ entries of a binary string drawn from a probability measure $\nu$ on $\{0, 1\}^{\bbN_0}$. 
    \item For each $t \in \bbN$, the denoiser $f_t \colon \bbR^n \to \bbR^n$ is separable and is given by
    \begin{align}
        f_{ti}(x_{t-1}) = \eta_{t}(x_{t-1,i}; \,  q_{\tau(t,i)}, \lambda r_{\tau(t,i)})  \label{eq:ft_eta}
    \end{align}
    where $\eta_t\colon \bbR \times \bbR_+ \times \bbR \to \bbR$ is a Lipschitz continuous scalar denoiser that is fixed for all $n$ and $\tau$ is the index function defined in \eqref{eq:tauDefinition}.
\end{enumerate}
\end{assumption}

According to our convention that every row is updated at the first time step ($t=0$), the measure $\nu$ is supported on binary strings whose first entry is one. For each $t \in \bbN$, we define $p_t$ to the probability mass function for the position of the last non-zero entry occurring  before time $t$, i.e., 
\begin{align}
    p_t(s) \coloneqq \nu_{t}\big( \{ \omega \in \{0,1\}^t : \omega_s = 1, \omega_{s+1} = \cdots = \omega_{t-1} = 0\}\big) \end{align}
for all $s \in \{ 0, 1, \dots, t-1\}$. This probability mass function can also be defined directly via the limiting empirical measure as 
\begin{align}
    p_t(s) \coloneqq \lim_{n \to \infty}  \frac{1}{n} \sum_{i=1}^n \one\{ \tau(t, i ) = s \}   \label{eq:limitingmeasure}
\end{align}
where Assumption~\ref{as:singLetterAssumptions}.1 ensures that the limit exists. 

We recall that $(q_t, r_t)$ correspond to the finite-$n$ overlap parameters of SE at iteration $t$ \eqref{eq:qtrt}. To define their asymptotic limits, we introduce the pairs $(\sigma^2_t, \rho_t) \in \bbR_+ \times \bbR$, that depend on the limiting measure $\mu \otimes \nu$ and the denoiser $\eta$ as follows. The initialization is given by 
\begin{align}
\sigma^2_0 = \int  \hat{u}^2   \,  d \mu(u, \hat{u})  ,  \qquad \rho_0 = \int  u \hat{u} \,  d \mu(u, \hat{u}) 
\end{align}
For $t \ge 1$, the update is given by
\begin{align} \label{eq:se_singleletter_updates}
    \sigma^2_t &=  \sum_{s<t} \psi_t(\sigma^2_s, \lambda \rho_s) \,  p_{t}(s), \qquad  \rho_t  =  \sum_{s<t} \phi_t(\sigma^2_s, \lambda \rho_s) \,  p_t(s)
\end{align}
where 
the functions $\psi_t, \phi_t : \bbR_+ \times \bbR \rightarrow \bbR_+$ are defined as
\begin{align}
    \psi_t( v,w) &\coloneqq  \int  \ex{( \eta_t( w u  + \sqrt{v}z ; v,w) )^2}  \, d\mu(u, \hat{u} ) \nonumber \\
    \phi_t( v,w) & \coloneqq  \int  \ex{ u \,   \eta_t( wu  + \sqrt{v} z; v,w) )  } \,  d\mu(u, \hat{u})  
\end{align}
The expectation is taken with respect to a standard Gaussian variable $z \sim \normal(0, 1)$. From the product assumption on the limiting empirical measure $\mu \otimes \nu$, it follows that, for any $t \in \bbN$, the asymptotic SE $(\sigma^2_t, \rho_t)$ corresponds to the large-$n$ limit of the SE parameters $(q_0,r_0)$:
\begin{align}
    \lim_{n \rightarrow \infty} q_t = \sigma^2_t, \quad \lim_{n \rightarrow \infty} r_t = \rho_t
\end{align}
We note that when applied to full matrix updates (i.e., $\delta_t \equiv \mathsf{1}$) these SE update equations recover the usual form the SE appearing in AMP analysis where $p_t(s)$ places all of its mass at the previous time step. 

\subsection{Power Iteration} \label{sec:powerIterSE}
In this section, we provide asymptotic guarantees on the limiting absolute empirical correlation of the projection AMP estimates $\{\hat{\theta}_t\}$ when the denoisers $f_t$ are chosen to be the projection onto the sphere of radius $\sqrt{n}$. The projection AMP algorithm takes the form
\begin{align} \label{eq:opampProjDenoiser}
    x_t = \delta_t \circ \Big(M \hat\theta_t - \frac{1}{\sqrt{n}\|x_{t-1}\|}\sum_{s<t} w_{ts} \hat\theta_s \Big) + (\mathsf{1}-\delta_t) \circ x_{t-1}, \quad \hat\theta_{t+1} = \frac{\sqrt{n}}{\|x_t\|}x_t
\end{align}
for $w_{ts} \coloneqq \sum_{i=1}^n \one\{\tau(t,i) = s\}$, the number of coordinates that were last updated at time $s$. We remark that this recursion encompasses as a special case the distributed power method \eqref{eq:opampPowerIter} up to a scaling factor on $\theta$. Assumption~\ref{as:singLetterAssumptions} implies that $(q_0,r_0)$ satisfy
\begin{align}
        q_0 \coloneqq \frac{1}{n}\|\hat\theta_0\|^2 \xrightarrow[n\rightarrow\infty]{} 1, \quad
        r_0 \coloneqq \frac{1}{n}\langle \theta, \hat{\theta}_0\rangle \xrightarrow[n\rightarrow\infty]{} \int u \hat{u}\, d\mu(u,\hat{u}) \eqqcolon \rho_0
\end{align}
Starting with $\rho_0$, then, we define the recurrence relation
\begin{align} \label{eq:PowMethSingleLetter}
    \rho_t = \frac{\lambda \sum_{s<t} \rho_s p_t(s)}{\sqrt{\sum_{s<t}(\lambda^2 \rho_s^2 + 1) p_t(s)}}
\end{align}
The following result relates the sequence $\{\rho_t\}$ and with the asymptotic correlation between the signal $\theta$ and the distributed power method estimates $\{\hat\theta_t\}$ described in \eqref{eq:opampProjDenoiser}, and is proved in Appendix~\ref{sec:powMethSingleLetterProof}.
\begin{theorem} \label{th:opampPMsingleLetter}
     Let $M$ be a spiked matrix model \eqref{eq:spikedMatrixUnitSignal}, $(\theta,\hat\theta_0)$ satisfy Assumption~\ref{as:singLetterAssumptions} and $\|\hat\theta_0\| = \sqrt{n}$. Consider the estimate sequence $\{\hat \theta_t\}$ produced by \eqref{eq:opampProjDenoiser}. Then, for each $t \in \bbN$, 
    \begin{align}
        \left| \frac{1}{n}\langle \theta, \hat \theta_t \rangle - \rho_t \right| \xrightarrow[n \rightarrow \infty]{\mathrm{p}} 0
    \end{align}
    for $\{\rho_t\}$ defined recursively as in \eqref{eq:PowMethSingleLetter}.
\end{theorem}

\begin{remark}
The above result is asymptotic in nature and, according to Assumption~\ref{as:singLetterAssumptions}, it requires a form of warm start $\hat\theta_0$ such that the limiting correlation $\rho_0$ is strictly positive. In finite dimensions, any random initialization will achieve non-zero overlap $r_0$ with high probability, and the AMP-adjusted power method described in~\eqref{eq:opampPowerIter} can be observed to converge to a non-trivial fixed point. The value of the corresponding fixed-point overlap can be approximated to a high degree of precision by considering the SE fixed point under a small initial perturbation of $\rho_0$ about zero.
\end{remark}

\subsection{Bayes-Optimal Denoising} 
One of the key strengths of the AMP framework is that prior information about the unknown signal can be incorporated naturally via the specification of the denoising functions. Throughout this section, we consider the Bayesian setting where the signal entries $\theta_1, \dots, \theta_n$ are drawn independently from a fixed probability measure $\pi$ on $\bbR$ with sub-Gaussian tails. Under this assumption,  the sequence of overlap parameters $\{(q_t, r_t)\}$ is random because it depends on the realization of $\theta$.  However, following the concentration arguments outlined in \cite{feng:2022}, it can be shown that these overlaps converge in the large-$n$ limit to a non-random limit described by the asymptotic SE parameters $\{(\sigma^2_t, \rho_t)\}$ defined by~\eqref{eq:se_singleletter_updates}. Note that this limit depends only on the signal distribution $\pi$ and the limiting empirical measure $\nu$ of the projector sequence. 

The practical upside of having an asymptotic SE that depends on the distribution $\pi$ (instead of $\theta$) is the that the the asymptotic overlap parameters $\{(\sigma^2_s,\rho_s)\}_{s <t}$ can be used to design the denoiser at time step $t$. In view of \eqref{eq:yts_model}, we see that each comparison vector $y_t$ can be seen as a collection of $n$ independent additive Gaussian noise channel observations of the signal entries $\theta_i$ given by 

\begin{align}
    y_{ti} = \lambda \rho_{\tau(t,i)} \theta_i + \sigma_{\tau(t,i)}z_i, \quad z_i \overset{\mathrm{iid}}{\sim} \normal(0,1).
\end{align}
For the prior $\theta \sim \pi^{\otimes n}$, the optimal estimator of $\theta$ under mean square error is the conditional mean estimator. This is  a separable function of the form given in \eqref{eq:ft_eta} with

\begin{align}
    \eta(y; v, w)  \coloneqq \frac{ \int u \exp\{ \frac{ w u y}{ v}   -  \frac{w^2 u^2}{ 2 v}   \} \, d \pi(u)   }{ \int \exp\{ \frac{w u y}{ v}   -  \frac{w^2 u^2}{ 2 v}   \} \,   d \pi(u)  }  \label{eq:cond_mean} 
\end{align}

From iterated expectations, it can be verified that $\psi(v,w) = \phi(v,w) = \Psi( w^2/v)$ for all $(v,w) \in \bbR_{+}\times \bbR$ where $\Psi \colon \bbR_+ \to \bbR_+$ is defined by 
\begin{align}
    \Psi( \gamma) \coloneqq \int \frac{\left(\int u \exp\{ \sqrt{\gamma} u y -  \frac{1}{2} \gamma u^2 \} \,   d \pi(u)\right)^2}{\int \exp\{ \sqrt{\gamma} u y -  \frac{1}{2} \gamma u^2 \} \, d \pi(u)} \frac{1}{\sqrt{2\pi}}e^{-\frac{1}{2}y^2} \, dy 
\end{align}
Consequently, the asymptotic overlap parameters satisfy $\rho_t = \sigma^2_t$ for all $t$. With this observation, we conclude that the asymptotic SE is described by the one-parameter recursion 
\begin{align} \label{eq:bayesSE}
\rho_t &=  \sum_{s<t} \Psi( \lambda^2 \rho_s) \,  p_{t}(s). 
\end{align}
Evaluating this recurrence relation in the special case where $p_t$ is a delta function at $t -1$ recovers the standard SE recursion for Bayes-optimal AMP with full matrix updates.

\subsection{Numerical Results} \label{sec:numerics}
To evaluate empirically our findings, we test different setups for the projection AMP recursion \eqref{eq:projAMP} acting on a rank-one spiked matrix model $M$ following \eqref{eq:spikedMatrix}. As denoiser sequences $\{f_t\}$, we apply both the sphere-projection function (corresponding to the OpAMP power method \eqref{eq:opampProjDenoiser}) and the separable conditional mean estimator from~\eqref{eq:cond_mean} corresponding to the Bayes optimal setting. For both choices of denoisers, we study the following protocols, each of which specifies how to select the erasure patterns $\delta_t$ (recall that a $0$ in the $i^{\text{th}}$ coordinate of $\delta_t$ corresponds to an erasure in the $i^{\text{th}}$ row of $M$ at the $t^{\text{th}}$ iteration), which in turn yield expressions for the limiting empirical measure $p_t(s)$ from~\eqref{eq:limitingmeasure}: 

\begin{itemize}
    \item \emph{Full-matrix updates.} At each iteration, we apply the entire data matrix. Thus, there are no erasures, corresponding to setting $\delta_t = \mathsf{1}$ (the all-ones vector) for all $t$ as well as
    \begin{align}
        p_t(s) = \begin{cases} 1 & \text{if } s = t-1,  \\ 0 & \text{if } s<t-1.\end{cases}
    \end{align}
    \item \emph{Random updates.} At each iteration, we sample the entries of $\delta_t$ IID from a $\mathsf{Bern}(\gamma)$-distribution, and apply the corresponding rows of the data matrix. This could correspond to the distributed computing setting where some straggling servers fail to respond by the deadline. We have that 
    \begin{align}
        p_t(s) = \begin{cases}
            \gamma (1-\gamma)^{t - s - 1} & \text{if } s=1,2, \dotsc, t-1, \\
            (1-\gamma)^{t - 1} & \text{if } s=0.
        \end{cases}
    \end{align}
    \item \emph{Round-robin updates.} We partition the set $\{1,\dotsc,n\}$ into $J$ equally-sized subsets $S_1,\ldots,S_J$. At the $t^{\text{th}}$ iteration, $\delta_t$ is one in entries $i \in S_{(t \mod{J})}$ and zero elsewhere. This could represent the setting where we schedule partial updates due to computational constraints. For the limiting measure, we distinguish two cases,
    \begin{subequations}
    \begin{align}
        &t < J : \quad p_t(s) = \begin{cases}
            1- \frac{t}{J} & \text{if } s = 0, \\
            \frac{1}{J} & \text{if } s = 1,2,\dotsc,t-1,
        \end{cases} \\
        &t \geq J : \quad p_t(s) = \begin{cases}
            0 & \text{if } s = 0, \dotsc, t-J-1, \\
            \frac{1}{J} & \text{if } s = t-J, \dotsc, t-1.
        \end{cases}
    \end{align}   
    \end{subequations}    
\end{itemize}

For all of our simulations, we set the dimension to $n = 5,000$, the signal strength to $\lambda = \sqrt{2}$, and average over 100 Monte Carlo trials, each with the same ground truth $\theta \in \{ \pm 1\}^n$ and initialization $f_0 \in \bbR^n$ with $\frac{1}{n} \langle \theta , f_0 \rangle = 0.01$. To equate the performance of the update protocols, we set $\gamma = 1/10$ and $J = 10$, corresponding to applying one-tenth of the data matrix at each iteration.

The performance of the sphere-projection denoiser was plotted earlier in Figure~\ref{fig:pi_intro_plots} as part of the introductory case study in Section~\ref{sec:casestudy}. (Note that the description in the introduction uses a different scaling in order to emphasize the connection with the power method. The parameter settings described in this section are also correct for Figure~\ref{fig:pi_intro_iter_plot}, due to the scale-invariant nature of the problem.) As one might expect, both the random update and the round-robin protocols evolve more slowly towards the common fixed point compared to the full-matrix protocol. It is also intuitive that round robin has an advantage compared to random update, as we are guaranteed that no coordinate lags more than $J$ iterations before an update. Similar trends can be seen for the Bayes-optimal denoiser in Figure~\ref{fig:bayes_iter}, and a careful comparison will reveal that it outperforms the suboptimal sphere-projection denoiser.

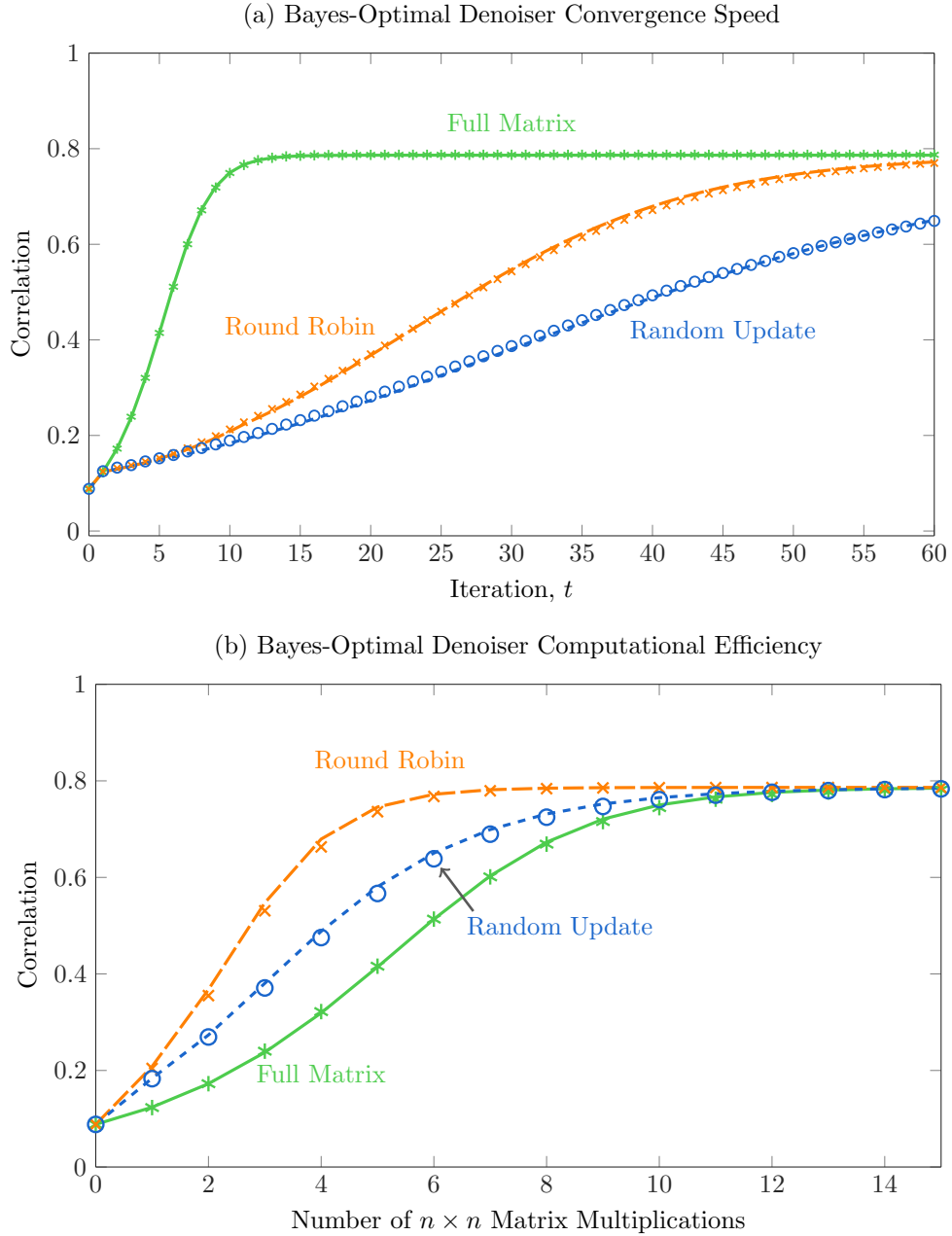
\begin{figure}[!ht]
\centering
\subfigure{
\label{fig:bayes_iter}
\def\datafilename{"fig_data/bayes_amp_iai.csv"} 

\begin{tikzpicture}
\begin{axis}[%
width=0.7\textwidth,
height=0.4\textwidth,
scale only axis,
separate axis lines,
every outer x axis line/.append style={white!15!black},
every x tick label/.append style={font=\color{white!15!black}},
xmin=1,
xmax=100,
xlabel={Iteration, $t$},
every outer y axis line/.append style={white!15!black},
every y tick label/.append style={font=\color{white!15!black}},
ymin=-0.01,
ymax=1,
ylabel={Correlation},
ylabel near ticks,
title = {(a) Bayes-Optimal Denoiser Convergence Speed},
legend style={at={(0.98,0.96)},anchor=south east,draw=white!15!black,fill=white,legend cell align=left}
]

\addplot [color=plot-green,solid, line width=1.4pt] table[col sep=comma, x="iter", y="se_full"] {\datafilename};

\addplot [datamarkstyle, color = plot-green, mark = asterisk]
            table [col sep = comma, x="iter",y="amp_full"] {\datafilename} node[pos=0.15, above = .1cm ] {Full Matrix};

\addplot [color=orange,dash pattern=on 8pt off 2pt, line width=1.4pt] table[col sep=comma, x="iter", y="se_robin"] {\datafilename};

\addplot [datamarkstyle, color = orange, mark = x]
            table [col sep = comma, x="iter",y="amp_robin"] {\datafilename} node[pos=0.15, above = .9cm ] {Round Robin};

\addplot [color=plot-darkblue,dashed,line width=1.4pt] table[col sep=comma, x="iter", y="se_rand"] {\datafilename}node[near end, below = .5cm ] {Random Update};

\addplot [datamarkstyle, color = plot-darkblue, mark = +]
            table [col sep = comma, x="iter",y="amp_rand"] {\datafilename}
            node[pos=0.4,below=.45cm] {Random Update};

\end{axis}
\end{tikzpicture}%
}
\subfigure{
\label{fig:bayes_comp}
\def\datafilename{"fig_data/bayes_amp_iai.csv"} 

\begin{tikzpicture}
\begin{axis}[%
width=0.7\textwidth,
height=0.4\textwidth,
scale only axis,
separate axis lines,
every outer x axis line/.append style={white!15!black},
every x tick label/.append style={font=\color{white!15!black}},
xmin=1,
xmax=15,
xlabel={Number of $n \times n$ Matrix Multiplications},
every outer y axis line/.append style={white!15!black},
every y tick label/.append style={font=\color{white!15!black}},
ymin=0,
ymax=1,
ylabel={Correlation},
ylabel near ticks,
title = {(b) Bayes-Optimal Denoiser Computational Efficiency},
legend style={at={(0.98,0.03)},anchor=south east,draw=white!15!black,fill=white,legend cell align=left}
]

\addplot [color=orange,dash pattern=on 8pt off 2pt, line width=1.4pt] table[col sep=comma, x="iter_norm", y="se_robin"] {\datafilename} node[pos=0.5, above = .1cm ] {Round Robin};
;

\addplot [datamarkstyleB, color = orange, mark = x]
            table [col sep = comma, x="iter_norm",y="amp_robin"] {\datafilename};

\addplot [color=plot-darkblue,dashed,line width=1.4pt] table[col sep=comma, x="iter_norm", y="se_rand"] {\datafilename} node[pos=0.37, above = .03cm, rotate=33 ] {Random Update};

\addplot [datamarkstyleB, color = plot-darkblue, mark = +]
            table [col sep = comma, x="iter_norm",y="amp_rand"] {\datafilename};

\addplot [color=plot-green,solid,line width=1.4pt] table[col sep=comma, x="iter", y="se_full"] {\datafilename};

\addplot [datamarkstyleB, color = plot-green, mark = asterisk]
            table [col sep = comma, x="iter",y="amp_full"] {\datafilename} node[pos=0.04, below =.1cm, rotate=33 ] {Full Matrix};

\end{axis}
\end{tikzpicture}%
}
\caption{Correlation $\frac{1}{n} \langle \theta, \hat{\theta}_t \rangle$ attained after $t$ iterations for the full matrix, round robin (updating $\frac{n}{10}$ rows per iteration), and random update (updating a row with probability $\frac{1}{10}$) protocols. The denoising functions $f_t$ are the separable, Bayes-optimal denoisers described in \eqref{eq:cond_mean}. All three protocols converge to the same fixed point and, as expected, full matrix updates requires fewer iterations to converge. Markers denote the average value over $100$ trials, standard errors were omitted as they were of negligible size. Solid and dashed curves denote the theoretically-predicted asymptotic overlap $\rho_t$ evaluated from~\eqref{eq:bayesSE} using the state evolution. }
\label{fig:bayes_plots}
\end{figure}

The surprising aspect of our numerical results is that, when evaluated in terms of their computational complexity, the partial-update protocols are more efficient than the full-update protocol. Specifically, we use the number of effective $n \times n$ matrix multiplications as a simple proxy for the required computation. (Of course, this ignores both the cost of the denoising function as well as the communication costs in a distributed setting, but may serve as a reasonable estimate when $n$ is large compared to total iterations.) For the full-matrix update, the number of matrix multiplications is equal to the iteration count $t$. For random updates, the expected number of matrix multiplications is $\gamma t$, and, for round-robin updates, the number of matrix multiplications is $t/J$. As observed in Figure~\ref{fig:pi_intro_comp_plot} for the sphere-projection denoiser and Figure~\ref{fig:bayes_comp} for the Bayes-optimal denoiser, the round-robin protocol is the most efficient with respect to the number of required matrix multiplications to reach a specified correlation. In many cases, we observe empirically that the random-update protocol outperforms the full-matrix update, but this is not always the case. As before, we note that the Bayes-optimal denoiser outperforms the sphere-projection denoiser when both follow the same update schedule.

\clearpage
\bibliographystyle{unsrt}
\bibliography{lib}

\appendix

\section{Full-Memory AMP}
\label{sec:full_memory}
In this section, we state and prove a convergence result for the full memory AMP recursion of the form
\begin{align}
    x_t = Z f_t(x_{< t} )- \sum_{s < t} b_{ts} f_s( x_{<s}) 
    \label{eq:AMPfullmemory}
\end{align}
This is the special case of the OpAMP recursion corresponding to $\cL_t(Z) = Z$ for all $t$. 
The sequence $\{x_t\}$ is compared with a zero-mean Gaussian process $\{y_t\}$, whose covariance is defined recursively by the SE update equation
\begin{align}
   \cov( y_s,y_t)  = \frac{1}{n}\ex{\langle f_r(y_{<r}), f_s(y_{<s}) \rangle}  \Id_n \ , \quad 0 \le s \le t  \label{eq:SEcovFullmemory}
\end{align}
Here, we recall that $\cov(y_0) = \frac{1}{n} \|f_0\|^2 \Id_n$ where $f_0\in \bbR^n$ is the deterministic initialization. 

\begin{theorem}\label{thm:AMPfullmemory}
Let $\{x_t\}$ be generated by \eqref{eq:AMPfullmemory} and let $\{y_t\}$ be the zero-mean Gaussian process defined by the SE. Suppose Assumption~\ref{ass:ft} holds,  $Z \sim \mathsf{GOE}(n)$, and 
the coefficients $\{b_{ts} : 0 \le s < t \}$ are given by
\begin{align}
        b_{ts} = \frac{1}{n} \gtr( \ex{  \mathsf{D}_s f_t(y_{<t}) } )  \label{eq:bts}
\end{align}
Then, for any fixed number of iterations $T$, there exists  a sequence (in $n$) of couplings between $x_{\scriptscriptstyle{\le T}}$ and $y_{\scriptscriptstyle{\le T}}$ such that
\begin{align}
    \frac{\|x_{\scriptscriptstyle{\le T}} - y_{\scriptscriptstyle{\le T}} \|}{\sqrt{n}} \xrightarrow[n \to \infty]{\mathrm{p}} 0
\end{align}
\end{theorem}

We remark that related convergence results for full-memory AMP have appeared previously in the literature. For example, the arguments outlined in \cite[Appendix~A]{Montanari:2022} can be used to establish convergence, over any fixed number of iterations,  with respect to a fixed sequence of pseudo-Lipschitz test functions. By the results in Appendix~\ref{sec:modes_convergence}, this mode of convergence is formally equivalent to the one given in Theorem~\ref{thm:AMPfullmemory}.
Our proof of Theorem~\ref{thm:AMPfullmemory} builds on the direct connection between the full-memory recursion in in \eqref{eq:AMPfullmemory} and the matrix-valued recursion studied in \cite{gerbelot:2023}.

\subsection{Prior Results for Matrix-valued Iterates}
\label{sec:matrix_valued_amp}
The SE for AMP with non-separable functions was proved by Berthier et al.~\cite{berthier:2020}, and then later extended to the setting of matrix-valued iterates by Gerbelot and Berthier~\cite{gerbelot:2023}. This section summarizes  \cite[Theorem~2]{gerbelot:2023} using the notation of this paper.
For $Z \sim \mathsf{GOE}(n)$, the matrix-valued AMP recursion  of interest given by
\begin{align}
 X_t = Z F_t(X_{t-1} )   - F_{t}(X_{t-1})B_t^\top  \label{eq:matrixAMP}
\end{align}
where $X_t \in \bbR^{n \times d}$ is the iteration at time $t$, each $F_t \colon \bbR^{n \times d} \to \bbR^{n \times d}$ is a deterministic function, and each $B_{t} \in \bbR^{d \times d}$ is a deterministic matrix. 
The assumptions on $\{F_t\}$ can be summarized as follows:
\begin{enumerate}
    \item Each $F_t \colon \bbR^{n \times d} \to \bbR^{n \times d}$ satisfies
    \begin{align}
    \|F_t(X) - F_t(Y)\|_F \le L \|X- Y\|_F \left( 1 + \left( \frac{\|X\|_F}{\sqrt{n}}\right)^{r-1}  + \left( \frac{\|Y\|_F}{\sqrt{n}}\right)^{r-1}  \right)   
\end{align}
for all $X, Y \in \bbR^{n \times d}$ where $L \ge 0$ and $r \ge 1$ are finite real numbers that do not depend on $n$. 

\item For all $s,t \in \bbN_0$ and $\Omega \in \psd^{2d}$ the following limit exists and is finite
\begin{align}
    \lim_{n \to \infty} \frac{ 1}{n} \ex{ F_s(Y_{s})^\top F_t(Y_{t})}
\end{align}
where $(Y_s , Y_t) \sim \normal( 0 , \Omega \otimes \Id_n)$.
\end{enumerate}
Under these assumptions, the asymptotic SE covariance $(\Sigma_{st} \in \psd^d : s,t \in \bbN_0) $ is defined  using the recursive construction
\begin{align}
    \Sigma_{00} & =
    \lim_{n \to \infty} \frac{ 1}{n}  F_0(\emptyset)^\top F_0(\emptyset); \qquad 
    \Sigma_{st}  =
    \lim_{n \to \infty} \frac{ 1}{n}  \ex{ F_s(Y_{s-1} )^\top F_t(Y_{t-1} )}, \qquad 0 \le s \le t
\end{align}
where $Y_{0}, \dots, Y_{t-1} \sim   \normal(0 ,\Sigma_{<t} \otimes \Id)$ and  $\Sigma_{<t}$ is the $dt \times dt$ psd block matrix defined by $(\Sigma_{rs} : 0 \le r,s < t) $. The matrices  $\{B_{t}\}$ appearing in the correction term in \eqref{eq:matrixAMP} are defined as a function of the asymptotic SE covariance according to  
\begin{align}
(B_t)_{kl}  =\frac{1}{n} \gtr\left(\ex{ \mathsf{D}_{l\to k} F_t(Y_{t-1})  }\right) , \qquad 1\le k, l \le d  \label{eq:BtGB23}
\end{align}
where $Y_{0}, \dots, Y_{t-1} \sim   \normal(0 ,\Sigma_{<t} \otimes \Id)$ and  $\mathsf{D}_{l\to k}$ is the Jacobian matrix of the $k$-th column of the output with respect to the $l$-th column of the input. Under these assumptions, \cite[Theorem~2]{gerbelot:2023} asserts that, for any  $T \in \bbN$, $r \ge 1$, and sequence of functions $\{\psi_n : (\bbR^{n \times d})^{(T+1)} \to \bbR\}$ satisfying 
\begin{align}
 \left| \psi_n(x)- \psi_n(y) \right| \le   \frac{ \|x- y\|}{\sqrt{n}} \left( 1 + \left( \frac{ \|x\|}{\sqrt{n}}\right)^{r-1} + \left( \frac{ \|y\|}{\sqrt{n}}\right)^{r-1} \right)  \qquad \forall x,y \in (\bbR^d)^{(T+1)}
\end{align}
the following convergence holds: 
\begin{align}
 \big| \psi_n(X_{\le T} ) - \ex{  \psi_n( Y_{\le T}) }\big|  \xrightarrow[n \to \infty]{\mathrm{p} } 0 \label{eq:convegenceGB23}
\end{align}

\subsection{Proof of Theorem~\ref{thm:AMPfullmemory}}
\label{proof:thm:AMPfullmemory}
We begin by establishing a link between the full memory AMP recursion \eqref{eq:AMPfullmemory} and matrix-valued recursion \eqref{eq:matrixAMP}.  
Under the assumptions of \cite[Thoerem~2]{gerbelot:2023}, this implies approximate Gaussianity of the AMP iterates in the sense of \eqref{eq:convegenceGB23}. We then show that the assumptions used in \cite[Theorem~2]{gerbelot:2023} can be relaxed such that the existence of a well-defined limit of the SE is replaced by a boundedness constraint, and the Gaussian process $\{y_t\}$  in \eqref{eq:convegenceGB23} is defined with respect to the finite-$n$ SE instead of the asymptotic SE. The final step in the proof follows from Theorem~\ref{thm:XnYnGausslim}, which shows that \eqref{eq:convegenceGB23} in implies that there exists a sequence of couplings such that $\|x_{\scriptscriptstyle{\le T}} - y_{\scriptscriptstyle{\le T}}\|/\sqrt{n}$ converges to zero in probability for any fixed $T \ge 0$.   
\subsubsection{From matrix-valued AMP to full memory AMP}
For a fixed number of iterations $T \in \bbN_0$, the full memory AMP recursion in \eqref{eq:AMPfullmemory} with iterates $x_t \in \bbR^n$ and functions $f_t \colon \bbR^{n \times t} \to \bbR^n$ can be embedded into the matrix-valued version of \eqref{eq:matrixAMP} with iterates $X_t \in \bbR^{n \times T}$ and $F_t \colon \bbR^{n \times T} \to \bbR^{n \times T}$ using 
    \begin{align}
X_t &= \begin{bmatrix}
    x_0 & x_1 & \hdots & x_{t} & 0 & \hdots & 0
\end{bmatrix} \\
F_t(X_{t-1}) &= \begin{bmatrix}
    f_0(\emptyset) & f_1(x_0) & \hdots & f_{t} (x_{<t}) & 0 & \hdots 0 
\end{bmatrix} 
\end{align}
This implies that the AMP iterations \eqref{eq:AMPfullmemory} up to time $T$ can be represented as the matrix equation
\begin{align}
    X_{t} = Z F_{t}(X_{t-1}) - F_{t-1}(X_{t-2}) B_{t}^\top 
\end{align}
for each $B_{t} \in \bbR^{T \times T}$, $t \leq T$, a strictly lower triangular matrix with entries
\begin{align}
(B_{t})_{sr}  = 
\begin{dcases}
    \frac{1}{n} \ex{ \trb{\mathsf{D}_{r} f_s(y_{<s})}  } &  \text{if $r< s\leq t$} \\
    0  & \text{otherwise} 
\end{dcases}
\end{align}

Here, we see that $X_{t-1} \equiv x_{< t} = \begin{bmatrix} x_0 \; \cdots \; x_{t-1} \end{bmatrix}$ contains the entire history of the recursion up to time $t-1$ and the $t$-th column of $F_{t}$ is given by $f_t$, in turn defining $x_{\leq t }$. For this special structure, the matrices $B_{t} \in \bbR^{T \times T}$ defined above satisfy the definition of the correction term \eqref{eq:BtGB23}, where we used $\mathsf{D}_r f_s$ to denote the Jacobian matrix of $f_s$ computed with respect to the $r$-th column of its input. Letting $b_{ts} = (B_t)_{ts}$ for $s < t$ recovers the form of the scalars $\{b_{ts}\}$ in \eqref{eq:bts} and completes the reduction.

\subsubsection{Convergence along any subsequence}
The prior work on AMP with non-separable functions \cite{berthier:2020, gerbelot:2023} places assumptions on the $f_t$ ensuring that the SE covariance has a well-defined limit. Specifically, for every $\Omega \in \psd^T$ and $0 \le s , t \le T$, the following limit exists and is finite
\begin{align}
    \lim_{n \rightarrow \infty} \frac{1}{n} \ex{\langle f_s(y_{<s}) , f_t(y_{<t})\rangle}
\end{align}
where the expectation is taken with respect to $y_{<T} \sim  \normal(0, \Omega \otimes \Id_n)$

In this section, we argue that these assumptions can be relaxed to require only that the sequence of (finite $n$) SE covariance matrices is uniformly bounded. Approximate Gaussianity of the AMP recursion is then established by showing that the normalized difference between the iterates $\{x_t\}$ and the Gaussian approximation $\{y_t\}$ converges to zero in the large-$n$ limit. 

This generalization is possible is because of the assumed Lipschitz continuity of each $f_t$ function, which ensures that, if $y_{< T} \sim \normal( 0, \Omega \otimes \Id_n)$ for $\Omega \in \psd^{T}$, then the element-wise overlap
\begin{align}
    \frac{1}{n} \ex{ \langle f_s(y_{<s} ),  f_t(y_{<t} ) \rangle}
\end{align}
is continuous function of $\Omega$. It follows that, for any given $T$, the sequence of overlaps $\Sigma_{\leq T}$ is defined on a compact subset of the cone positive semidefinite matrices. If the overlaps are bounded uniformly in $n$, it is then possible to partition any sequence of problems into a finite number of subsequences, each of which is eventually contained within an $\eps$-ball centered about a fixed quantization point. By continuity, closeness of SE covariance implies closeness of the expectations under the class of pseudo-Lipschitz functions. Hence, we can apply the result for converging sequences in \cite[Theorem~2]{gerbelot:2023} to conclude that 
\begin{align}
 \big| \psi_n( x_{\leq T}) - \ex{  \psi_n(y_{\leq T}) }\big|  \xrightarrow[n \to \infty]{\mathrm{p} } 0 
\end{align}
for any sequence of uniformly pseudo-Lipschitz functions where $y_{\leq T} \sim \normal( 0, \Sigma_{\le T} \otimes \Id)$ is defined with respect to the finite-$n$ SE covariance. Finally, we appeal to the equivalence result in Theorem~\ref{thm:XnYnGausslim} to conclude that
\begin{align}
    \frac{\| x_{\scriptscriptstyle{\leq T}} - y_{\scriptscriptstyle{\leq T}} \|}{\sqrt{n}} \xrightarrow[n \rightarrow \infty]{\mathrm{p}} 0 
\end{align}
This concludes the proof Theorem~\ref{thm:AMPfullmemory}.
\qed

\section{Proofs of Main Results} 
\subsection{Proof of Theorem~\ref{thm:opAMP}}
The high-level idea of this proof is express  the OpAMP recursion $\{x_t\}$ in terms of a ``lifted recursion'' $\{w_{tk}\}$, indexed by $t \in \bbN$ and $k =1, \dots, K$, according to 
\begin{align}
x_t = \sum_{k=1}^K L_{tk} w_{tk}   \label{eq:wtox}
\end{align}
where the $n \times n$ matrices  $L_{tk}$ are defined by the decomposition of $\cL_t$. We then show that the lifted recursion is a doubly-indexed version of the basic full memory AMP recursion in \eqref{eq:basicAMP}.  
Making appropriate adjustments for the indexing, we apply  Theorem~\ref{thm:AMPfullmemory} to obtain a SE and  Gaussian approximation for  $\{w_{tk}\}$. The desired result for the OpAMP recursion then follows straightforwardly from the linearity of the mapping from  $\{w_{tk}\}$ to $\{x_t\}$ given in \eqref{eq:wtox}.  

\subsubsection{Lifted recursion}
We begin by introducing the recursion $\{w_{tk} \in\bbR^n\}$ indexed by the pair $(t,k)$ with $t \in \bbN$ and $k = 1,\ldots,K$. This recursion is defined by  a sequence functions $\{g_{tk}\}$ and a sequence of scalar debiasing terms $\{c_{tksl}\}$ according to parallel updates of the form 
\begin{align}
w_{tk} &= Z g_{tk}(w_{< t}) -  \sum_{s < t} \sum_{l=1}^K c_{tksl} \,  g_{sl}(w_{< s}) \ , \quad \text{for $k=1,\dotsc,K$} \label{eq:wtk}
\end{align}
where we use the shorthand notation $w_{< t} = (w_{sl} : s < t, 1 \le l \le K)$ to represent the entire history of the $\{w_{tk}\}$ sequence up to time step $(t-1)$. Each function $g_{tk}$ is defined in terms the function $f_{t}$, the matrix $R_{tk}$,  and the collection of matrices $L_{<t} = ( L_{sl} : s< t, 1 \le l \le K)$ according to 
\begin{align}
    g_{tk}(w_{<t} ) & = R_{tk}f_t\Big(\sum_{k=1}^K L_{0k} w_{0k}, \ldots, \sum_{k=1}^K L_{t-1,k} w_{t-1, k}  \Big) \label{eq:gtk}
\end{align}

Next, we observe that the recursion defined by \eqref{eq:wtk} is a doubly-indexed version of the full-memory AMP recursion in \eqref{eq:AMPfullmemory}. To make this connection explicit, one may consider the bijective mapping between the index pair $(t,k)$ and a single iteration index $\tau \in \bbN_0 $ given by $\tau = tK +k$. 
Under the assumptions of Theorem~\ref{thm:opAMP}, it it not too difficult to see that the collection of functions $\{g_{tk}\}$ satisfies Assumption~\ref{ass:ft}, i.e., each $g_{tk}$ has a Lipschitz constant that is bounded uniformly for all $n$. Consequently, we can apply Theorem~\ref{thm:AMPfullmemory} to establish the existence of a coupling between $\{w_{tk}\}$ and a zero-mean Gaussian process  $\{u_{tk}\}$ whose covariance is defined by SE equations. 

In order to describe the SE for the the lifted recursion, we follow the specification given in Theorem~\ref{thm:AMPfullmemory}, making appropriate adjustments to handle the double indexing. Specifically, we define $\{u_{tk}\}$ be a zero-mean Gaussian process whose covariance is defined recursively by the SE update equation 
\begin{align}
    \cov(u_{sl}, u_{tk}) = \frac{1}{n} \ex{ \langle g_{sl}(u_{<s}), g_{tk}(u_{<t})} \,  \Id_n , \quad 0 \le s \le t, \quad k,l  =1, \dots , K
\end{align}
The debiasing terms $\{c_{tksl} : 0 \le s < t, k,l = 1,\dots, K\}$ are given by
\begin{align}
        c_{tksl} = \frac{1}{n} \gtr( \ex{  \mathsf{D}_{sl} g_{tk}(u_{<t}) } )  \label{eq:ctksl}
\end{align}
where $\mathsf{D}_{sl}$ denotes the Jacobian of $g_{tk}$ computed with respect to the input vector indexed by $(s,l)$.

\subsubsection{Mapping back to OpAMP}
The OpAMP recursion in \eqref{eq:opAMP} can be recovered from the lifted recursion in \eqref{eq:wtk} by applying the transformation in \eqref{eq:wtox}. To see this, we first make the the substitution $g_{sl}(w_{<s} ) = R_{sl}f_s(x_{s-1})$ for $s \le t$ and $l = 1, \dots , K$,  which leads to 
\begin{align}
w_{tk} &= Z R_{tk}f_t(x_{t-1}) -  \sum_{s < t} \sum_{l=1}^K c_{tksl} \, R_{sl} f_s(x_{s-1}) \ , \quad \text{for $k=1,\dotsc,K$} 
\end{align}
Applying \eqref{eq:wtox}  and recalling that $\cL_t(Z) = \sum_{t=1}^K L_{tk}Z R_{tk}$ then leads to the OpAMP recursion in \eqref{eq:opAMP}, where the  debiasing matrices are given by
\begin{align}
    B_{ts} = \sum_{k,l= 1}^K c_{tksl} L_{tk} R_{sl}  \label{eq:c_to_B}
\end{align}

The SE for OpAMP follows from applying the same transformation to SE of lifted recursion. Specifically, let $\{u_{tk}\}$ be the zero-mean Gaussian process described above and let $\{y_t\}$ be the zero-mean Gaussian process defined by $y_t = \sum_{k=1}^K L_{tk} u_{tk}$. The covariance satisfies 
\begin{align}
    \cov( y_s, y_t) = \sum_{k,l=1}^L L_{sl} \cov( u_{sl}, u_{tk}) L_{tk}^\top \label{eq:covutocovy}
\end{align}
Making the substitution $g_{sl}(u_{<t}) = R_{sl} f_s(y_{s-1})$ leads to an expression that depends only the previous terms in the $\{y_t\}$ sequence:
\begin{align}
    \cov( u_{sl}, u_{tk}) & = \underbrace{\frac{1}{n} \ex{ \langle R_{sl} f_s(y_{s-1}), R_{tk} f_t(y_{t-1})}}_{\eqqcolon  q_{sltk}} \,  \Id_n
\end{align}
Plugging this expression back into \eqref{eq:covutocovy} recovers the SE covariance update in \eqref{eq:SEopAMP}.

Finally, we derive an expression for the debiasing coefficients $\{c_{tksl}\}$ in terms of $\{y_t\}$. By \eqref{eq:gtk} and the chain rule for differentiation, we see that
\begin{align}
    \frac{\partial}{ \partial u_{sl}}  g_{tk}(u_{<t}) 
    &  =    \frac{\partial}{ \partial u_{sl}}R_{tk}  f_t\Big(\sum_{k=1}^K L_{0k} u_{0k}, \ldots, \sum_{k=1}^K L_{t-1,k} u_{t-1, k}  \Big)    =  R_{tk}  \frac{\partial}{ \partial y_{s}} f_t(y_0, \dots, y_{s-1} ) L_{sl}  
\end{align}
In conjunction with \eqref{eq:ctksl}, we obtain
\begin{align}
    c_{tksl}  & = \frac{1}{n} \gtr\left( R_{tk} \ex{\frac{ \partial}{\partial y_{s} } f_t(y_0,\ldots,y_{t-1})} L_{sl}   \right) 
\end{align}
Plugging this expression back into \eqref{eq:c_to_B} recovers the correction term given in Theorem~\ref{thm:opAMP}. 
This concludes the proof of Theorem~\ref{thm:opAMP}.
\qed

\subsection{Proof of Theorem~\ref{thm:opAMPlinmemory}}
The first step of this proof is to use a change of variable to express the autoregressive iteration
\begin{align}
    x_{t} = \mathcal{L}_t(Z)f_t(x_{t-1}) + \sum_{s<t} A_{ts} x_s - \sum_{s<t} B_{ts}f_s(x_{s-1}) \label{eq:amp_aut2}
\end{align}
as a special case of the the OpAMP recursion in \eqref{eq:opAMP}  To this end, we introduce the modified sequence  $\{\tilde{x}_t\}$ according to 
\begin{align}
    \tilde{x}_t & = x_t - \sum_{s < t} A_{ts} x_s  \label{eq:tildexttoxt}
\end{align}
The transformation from $\{x_t\}$ to $\{\tilde{x}_t\}$ is invertible with inverse mapping given by 
\begin{align}
       x_t & =  \sum_{s \le t} C_{ts} \tilde{x}_s  \label{eq:xttotildext}
\end{align}
where  the collection of $n \times n$ matrices $\{C_{ts}\}$ is defined by \eqref{eq:AtoC} and we have used the convention $C_{tt} = \Id_n$ to simplify the presentation.  Applying this change of variable to \eqref{eq:amp_aut2}, one finds that the $\{\tilde{x}_t\}$ process can be expressed in terms of the OpAMP recursion: 
\begin{align}
    \tilde{x}_t 
    & = \tilde{f}_t( \tilde{x}_{< t} ) - \sum_{s < t} B_{ts} \tilde{f}_{s}( \tilde{x}_{< s})\label{eq:tildext}
\end{align}
where each $\tilde{f}_t \colon \bbR^{n \times T} \to \bbR^n$ is given by 
\begin{align}
\tilde{f}_t(\tilde{x}_{<t}) \coloneqq  f_t\Big(\sum_{s< t} C_{t-1,s} \tilde{x}_s \Big)
\end{align}
The next step is to verify the recursive construction for the SE given in \eqref{eq:SEopAMPlinmemory}. In the following, we use $\{y_t\}$ and $\{\tilde{y}_t\}$ to denote  the zero-mean Gaussian processes associated with the recursions $\{x_t\}$ and $\{\tilde{x}_t\}$, respectively.  From the bijection between $\{x_t\}$ and $\{\tilde{x}_t\}$ defined by \eqref{eq:tildexttoxt} and \eqref{eq:xttotildext}, these processes satisfy  
\begin{align}
       y_t & = \sum_{t \le t} C_{ts} \tilde{y}_s \quad \iff \quad \tilde{y}_t  = y_t - \sum_{s < t} A_{ts} y_s \label{eq:tildeyttoyt}
\end{align}
The covariance of $\{\tilde{y}\}$ is specified according to the recursive update equation of OpAMP given in \eqref{eq:SEopAMP}, i.e.,
\begin{align}
\cov(\tilde{y}_s, \tilde{y}_t)  & = \sum_{l,k=1}^K  q_{sltk}  L_{sl} L_{tk}^\top,\\
q_{sltk} & =  \frac{1}{n} \ex{\langle R_{sl} \tilde{f}_{s}(\tilde{y}_{<s}  ), R_{tk} \tilde{f}_{t}(\tilde{y}_{<t}) \rangle } 
\end{align}
By the correspondence $\tilde{f}_{t}(\tilde{y}_{<t}) = f_t(y_{t-1})$, the overlap terms can also be expressed in as expectation with respect to the $\{y_t\}$ processes:
\begin{align}
      q_{sltk}
      & = \frac{1}{n} \ex{\langle R_{sl} f_{s}(y_{s-1} ), R_{tk} f_{t}(y_{t-1}) \rangle }.  
\end{align}
Consequently,  the update equation for the covariance of $\{y_t\}$ given in \eqref{eq:SEopAMPlinmemory} is obtain by writing
\begin{align}
    \cov(y_s, y_t)
    &= \sum_{ s'\le s, t' \le t} C_{ss'} \cov\left(  \tilde{y}_{s'},  \tilde{y}_{t'}  \right)C_{tt'}^\top \\
       &= \sum_{s' \leq s, t' \leq t} \; C_{ss'} \Big(\sum_{l,k=1}^K q_{s'lt'k} L_{sl}L_{tk}^\top\Big) C_{tt'}^\top \\
    &= \sum_{s' \leq s, t' \leq t} \; \sum_{l,k=1}^K q_{s'lt'k} C_{ss'} L_{sl} (C_{tt'}L_{tk})^\top
\end{align}
Finally, we derive the form of the correction term given in \eqref{eq:BtsopAMPlinmemory}.  Following the specification for OpAMP in \eqref{eq:BtsopAMP}, the matrices $\{ B_{ts} : 0 \le s <t \}$ are given in terms of the $\{\tilde{y}_t\}$ process by
\begin{align}
B_{ts} & = \sum_{k,l=1}^K   \frac{1}{n} \gtr\left( R_{tk} \ex{\mathsf{D}_{s} \tilde{f}_t(\tilde{y}_{<t})} L_{sl} \right) L_{tk} R_{sl} \label{eq:Bts_aut2}
\end{align}
To express the correction in terms of the $\{y_t\}$ process, we use the definition of $\tilde{f}_t$ and the chain rule for differentiation to write
\begin{align}
\mathsf{D}_{s}\tilde{f}_t(\tilde{y}_{<t}) 
& = \frac{ \partial}{\partial \tilde{y}_s} f_t\Big(\sum_{t'< t} C_{t-1,t'} \tilde{y}_{t'} \Big)  = \mathsf{D} f_t\Big(\sum_{t'< t} C_{t-1,t'} \tilde{y}_{t'} \Big) C_{t-1,s}  = \mathsf{D} f_t(y_{t-1} ) C_{t-1,s} 
\end{align}
where the last step follows from \eqref{eq:tildeyttoyt} Plugging this expression back into \eqref{eq:Bts_aut2} leads to \eqref{eq:BtsopAMPlinmemory}. This concludes the proof of Theorem~\ref{thm:opAMPlinmemory}.
\qed

\subsection{Proof of Theorem~\ref{thm:projAMP}}
This result follows from noting that projection AMP is a special case of the autoregressive OpAMP recursion in \eqref{eq:opAMPlinmemory}, with $\cL_t(Z) = \Pi_{t} Z$ and linear memory terms given by
\begin{align}
    A_{ts} & = \begin{dcases}
        \Pi_{t-1} & s=t-1 \\
        0 & s < t-1
    \end{dcases}
\end{align}
The simplified structure of $C_{ts}$ follows from \eqref{eq:AtoC}. Under this specification, the matrices $B_{ts}$ defined by \eqref{eq:BtsopAMPlinmemory} can be expressed as $B_{ts} = \Pi_t b_{ts}$
where the scalars $b_{ts}$ are given by \eqref{eq:bts_projAMP}. Having established this connection, the desired convergence follows directly from the Theorem~\ref{thm:opAMPlinmemory} and the fact that projection matrices satisfy Assumptions~\ref{ass:Lt} and \ref{ass:Ats}.
\qed

\section{Equivalence of Different Modes of Convergence} \label{sec:modes_convergence}
This section provides some further results comparing various types of convergence guarantees.  

\begin{theorem}\label{thm:XnYnlim}
Let $\{X_n\}$ and $\{Y_n\}$ be two sequences of random elements with $X_n$ and $Y_n$ taking values in a complete separable normed space  $(S_n, \|\cdot\|_n)$ for each $n$. For $r \ge 1$ define the class of functions
\begin{align}
\mathsf{PL}_n(r)& = \left\{ f \colon S_n \to \bbR : (\forall x, y \in S_n)  (\exists L \in [0,\infty)) \quad |f(x) - f(y) | \le \|x- y\|_n ( 1 + \|x\|_n^{r-1} + \|y\|_n^{r-1} ) \right\}
\end{align}
Then, the implications (i) + (ii) $\implies$ (iii) $\implies$ (i) hold for the following statements:
\begin{enumerate}
  \item   There exists a sequence of couplings on $(X_n, Y_n)$ such that 
    \begin{align}
\|X_n - Y_n\|_n \xrightarrow[n \to \infty]{\mathrm{p} } 0 \label{eq:XnYnplim}
\end{align}
\item  $\{X_n\}$ (or $\{Y_n\}$) satisfies 
    \begin{align}
           \lim_{t \to \infty} \limsup_{n \to \infty}  \pr{ \|X_n\|_n \ge t}   = 0 \label{eq:Xntail}
    \end{align}

\item There exists a sequence of couplings on $(X_n, Y_n)$ such that, for every $r \ge 1$ and every sequence of functions $\{f_n \in \mathsf{PL}_n(r)\}$, 
    \begin{align}
| f(X_n) - f(Y_n)|  \xrightarrow[n \to \infty]{\mathrm{p} } 0 \label{eq:fXnfYnplim}
\end{align}
\end{enumerate}
\end{theorem}
\begin{proof}
[Proof: (i)  + (ii) $\implies$ (iii)] Let $\{(X_n, Y_n)\}$ be a sequence of couplings such that \eqref{eq:XnYnplim} holds. We begin by showing that $\{Y_n\}$ also satisfies the tail condition given in \eqref{eq:Xntail}. 
For any $t \ge 0$, the triangle inequality and the union bound give
\begin{align}
    \pr{ \|Y_n\|_n \ge t}  \le   \pr{ \|X_n - Y_n\|_n \ge t/2} +  \pr{ \|X_n\|_n \ge t/2}
\end{align}
Hence, by \eqref{eq:XnYnplim} and \eqref{eq:Xntail}, 
\begin{align}
    \lim_{t \to \infty} \limsup_{n \to \infty}  \pr{ \|Y_n\|_n \ge t}   = 0 \label{eq:Yntail}
\end{align}

Next, recall that by definition,  each $f \in \mathsf{PL}_n(r)$ satisfies
\begin{align}
    |f(X_n) - f(Y_n)|\le \| X_n - Y_n\|_n ( 1 + \|X\|_n^{r-1} + \|Y\|_n^{r-1} )
\end{align}
For any $\eps > 0$, the tail conditions \eqref{eq:Xntail} and \eqref{eq:Yntail} ensure that, for any choice of $r \geq 1$, there exist finite positive numbers $N$ and $t$  such that, for all $n \ge N$, 
\begin{align}
\pr{ \|X_n\|_n^{r-1}  +  \|Y_n\|_n^{r-1} > t}  \le \pr{ \|X_n\|_n^{r-1} > t/2}  + \pr{ \|Y_n\|_n^{r-1} > t/2} \le \eps 
\end{align}
Furthermore, by \eqref{eq:XnYnplim}, there exists a finite positive number $N'$ such that, for all $n \ge N'$, 
\begin{align}
  \pr{ \| X_n - Y_n\|_n > \frac{\eps}{1+2t} } \le \eps
\end{align}
By the union bound, it follows that for all $n \ge\max\{ N, N'\} $. 
\begin{align}
 \pr{ |f(X_n) - f(Y_n)| > \eps} \le \pr{ \| X_n - Y_n\|_n > \frac{\eps}{1+2t} } +  \pr{ \|X_n\|_n^{r-1}  +  \|Y_n\|_n^{r-1} > t} \le  2\eps
\end{align}
and thus condition (iii) holds.
\end{proof}

\begin{proof}[Proof: (iii) $\implies$ (i)] Let $\{(X'_n, Y'_n)\}$ be a sequence of couplings such that \eqref{eq:fXnfYnplim} holds. For each $n$, define the class of bounded-Lipschitz functions
\begin{align}
\mathsf{BL}_n& \coloneqq \big\{ f \colon S_n \to \bbR :  \quad (\forall x, y \in S_n) \quad |f(x) - f(y) | \le  \min\left\{  \|x-y\|_n, 1 \right\}  \big\}
\end{align}
By the the  Kantorovich-Rubinstein theorem \cite{villani:2009}, the minimum over all couplings w.r.t.\ to the metric $\min\left\{  \|x-y\|_n, 1 \right\}$ is equal to the integral probability metric defined by the class $\mathsf{BL}_n$, that is
\begin{align}
    \min_{\text{couplings of } (X_n,Y_n) } \ex{  \min\left\{\|X_n- Y_n\|_n, 1 \right\}}  & = \sup_{ f \in \mathsf{BL}_n}  \left|\ex{ f(X_n)}  -  \ex{ f(Y_n)}     \right|
\end{align}
To bound the difference in expectations on the RHS, we can fix any $\eps > 0$ and use the coupling $(X'_n, Y'_n)$ to write
\begin{align}
\left|\ex{ f(X_n)}  -  \ex{ f(Y_n)}     \right| & = \left|\ex{ f(X'_n) -  f(Y'_n)}     \right| \\
& \le  \left|\ex{ (f(X'_n) - f(Y'_n))\one_{|f(X'_n)-f(Y'_n) | \le \eps}   } \right| \nonumber\\
& \quad + \left| \ex{ (f(X'_n) - f(Y'_n))\one_{|f(X'_n)-f(Y'_n) | > \eps }   }   \right|  \\
    &\le \eps +  \pr{|f(X_n')-f(Y_n') | > \eps }  \\
    &\le \eps +  \sup_{f \in \mathsf{PL}_n(r) } \pr{|f(X'_n)-f(Y'_n) | > \eps }
\end{align}
where the last step follows from  the inclusion $\mathsf{BL}_n \subset \mathsf{PL}_n(r)$. By  \eqref{eq:fXnfYnplim}, the second term on the RHS converges to zero as $n \to \infty$. Since $\eps > 0$ is arbitrary, this means that there exists a sequence of coupling $\{(X_n, Y_n)\}$ such that $
    \ex{  \min\left\{\|X_n- Y_n\|_n, 1 \right\}} \to 0$. By Markov's inequality, this implies \eqref{eq:XnYnplim} and so (i) holds. 
\end{proof}

\begin{theorem}\label{thm:XnYnGausslim}
Let $\{X_n\}$ and $\{Y_n\}$ be two sequences of $n \times d$ random matrices where $Y_n \sim \normal( \Sigma_n \otimes \Id)$ with $\sup_n \|\Sigma_n\|_{op} < \infty $.
The following statements are equivalent: 
\begin{enumerate}
    \item There exists a sequence of couplings on $(X_n, Y_n)$ such that 
    \begin{align}
\frac{ \|X_n - Y_n\|_F}{\sqrt{n}} \xrightarrow[n \to \infty]{\mathrm{p} } 0
\end{align}

\item For every $r \ge 1$ and every sequence of functions $\{ \psi_n \colon \bbR^{n \times d} \to \bbR\} $ such that
\begin{align}
  |\psi_n(X) - \psi_n(Y)| \le \frac{  \|X- Y\|_F}{\sqrt{n}} \left( 1 + \left( \frac{\|X\|_F}{\sqrt{n}}\right)^{r-1}  + \left( \frac{\|Y\|_F}{\sqrt{n}}\right)^{r-1}  \right) \qquad \forall X, Y \in \bbR^{n \times d}  \label{eq:psi_n}
\end{align}
the following convergence holds:
\begin{align}
 \left| \psi_n( X_n) - \ex{  \psi_n(Y_n) }\right |  \xrightarrow[n \to \infty]{\mathrm{p} } 0 \label{eq:XnYnpsilim}
\end{align}
\end{enumerate}
\end{theorem}
\begin{proof}
We use Gaussian concentration to show  that the convergence in \eqref{eq:XnYnpsilim} implies the convergence given in \eqref{eq:fXnfYnplim}, where $\mathsf{PL}_n(r)$ the pseudo-Lipschitz function class defined with respect to $(S_n,\|\cdot\|_n)  = (  \bbR^{n \times d},  \| \cdot\|_F / \sqrt{n})$. To see this, observe that for any coupling  $X_n$ and $Y_n$ and for any $\eps > 0$, the triangle inequality yields
\begin{align}
 \pr{ \left| \psi_n( X_n) -   \psi_n(Y_n) \right | > \eps} & \le \pr{ \left| \psi_n( X_n) -   \ex{ \psi_n(Y_n) } \right | > \eps} +  \pr{ \left| \psi_n( Y_n) -   \ex{ \psi_n(Y_n)} \right | > \eps} 
\end{align}
Condition \eqref{eq:XnYnpsilim} implies that the first term on the r.h.s. converges to zero.  By the  Gaussian Poincaré inequality and the assumptions on $\psi_n$, it can be verified that   $\var( \psi_n(Y_n)) \le C(r)/n$ and thus the second term also converges to zero. Hence, condition (ii) implies that for any sequence of couplings and any sequence $\{\psi_n \colon \bbR^{n \times d} \to \bbR\}$ satisfying \eqref{eq:psi_n},  
\begin{align}
 \left| \psi_n( X_n) -  \psi_n(Y_n) \right |  \xrightarrow[n \to \infty]{\mathrm{p} } 0 
\end{align}

The desired results now follow straightforwardly from Theorem~\ref{thm:XnYnlim}, along  with the observation that the second moment constraint $ \ex{ \|Y_n\|_F^2} = n \|\Sigma_n\|^2_F$ combined with Markov's inequality implies the tail condition
\begin{align}
 \lim_{t \to \infty}   \limsup_{n\to \infty} \pr{ \frac{\|Y_n\|_F}{\sqrt{n}} > t} = 0 
\end{align}
\end{proof}

\section{Centering for the Spiked Matrix Model} \label{sec:AMPcentering}

This section outlines the centering arguments used for the spiked matrix model in Section~\ref{sec:stochastic_updates}. 
We recall the spiked matrix model and the projection AMP algorithm for the reader's convenience. 
The data matrix is a symmetric $n \times n$ matrix drawn from the spiked matrix model
\begin{align}
    M = \frac{\lambda}{n} \theta \theta^\top + Z; \quad Z \sim \mathsf{GOE}(n),
\end{align}
where $\theta \in \bbR^n$ is a fixed but unknown signal. A sequence  of estimate-iterate pairs  $\{(\hat{\theta}_t, x_t)\}$ is generated according to the projection AMP algorithm
\label{eq:thetatxt_proof}
    \begin{align}
    \hat{\theta}_t &= f_t(x_{t-1}) \, \\
    x_t &= \delta_t \circ \Big(M\hat{\theta}_t - \sum_{s<t} b_{ts} \hat{\theta}_s \Big ) + (\mathsf{1} - \delta_t) \circ x_{t-1}\ ,
\end{align}
with initialization $f_0 \equiv \hat{\theta}_0$. 

Our goal is to show that the sequence $\{x_t\}$ can be compared with the Gaussian sequence $\{y_t\}$, whose  (nonzero) mean and  covariance  are defined by the SE in \eqref{eq:spikedSE}. 
Setting $\mu_t = \ex{y_t}$ and $\Sigma_t = \cov(y_t)$, these SE update equations can be expressed compactly as
\label{eq:spikedSEcovProof}
\begin{align}
    \mu_t &=  \lambda r_t \delta_t \circ \theta + (\mathsf{1} - \delta_t) \circ \mu_{t-1}  \\
    \Sigma_t &=  \lambda q_t \diag(\delta_t) + (\Id_n - \diag(\delta_t)) \Sigma_{t-1}  .
\end{align}
where
\begin{align}
    q_t = \frac{1}{n} \ex{\|f_t(y_{t-1})\|^2} , \qquad  r_t = \frac{1}{n} \ex{\langle \theta, f_t(y_{t-1})}
\end{align}

The main difference between the projection AMP algorithm in \eqref{eq:spikedSEcovProof} and the recursion appearing in Theorem~\ref{thm:projAMP} is that the the data matrix $M$ has a nonzero mean. In order to apply the results in Theorem~\ref{thm:projAMP} we apply  centering argument that is standard in the AMP literature; see \cite{feng:2022}. Specifically, we show that that \eqref{eq:spikedSEcovProof} can be approximated by a modified recursion that is applied directly to the noise matrix $Z$. We emphasize that this new recursion is introduced primarily as proof technique. The centering transformation cannot be applied in practice, because it depends on the unknown signal $\theta$.  

While a variety of centering arguments can be applied, it is convenient to define the centered recursion in terms of the mean vector $\mu_t$ defined by \eqref{eq:spikedSEcovProof}. We consider the recursion on iterates $\{\tilde{x}_t\}$ given by
\begin{align}
    \tilde{x}_t = \delta_t \circ \Big(Z \tilde{f}_t( \tilde{x}_{t-1}) - \sum_{s<t} b_{ts} \tilde{f}_s(\tilde{x}_{s-1})  \Big) + (\mathsf{1} - \delta_t) \circ \tilde{x}_{t-1}
\end{align}
where $ \tilde{f}_t \colon \bbR^{n} \to \bbR^n$ is defined in terms of the pair $(f_t, \mu_t)$ according to 
\begin{align}
   \tilde{f}_t(\tilde{x}_{t-1} )  & = f_t(  \mu_t  +  \tilde{x}_{t-1}) 
\end{align}
This recursion satisfies the assumptions of Theorem~\ref{thm:opAMPlinmemory} and can thus be coupled with a zero-mean Gaussian process $\{\tilde{y}_t\}$, whose covariance satisfies the SE update equation in \eqref{eq:SEopAMPlinmemory}. Specializing these updates using the structure of the projection matrices, one finds that the covariance of $\{\tilde{y}_t\}$ matched to the covariance of the process $\{y_t\}$.  In particular, $\cov(\tilde{y}_t) =\Sigma_t$ for all $t$. 

From here, the remaining step is to argue that the shifted sequence defined by $\{ \mu_t + \tilde{x}_{t}\}$ is close in distribution to the sequence $\{x_t\}$ generated by \eqref{eq:thetatxt_proof}. This can be proved recursively using the approach described in \cite[Appendix~D]{gerbelot:2023} and leveraging the equivalence between convergence with respect to sequences of pseudo-Lipschitz test functions and the existence of a Gaussian coupling we show in Theorem~\ref{thm:XnYnGausslim}.

\section{Proof of Theorem~\ref{th:opampPMsingleLetter}} \label{sec:powMethSingleLetterProof}
Technically speaking, the sphere-projection denoisers in \eqref{eq:opampProjDenoiser} are not separable and the correction terms of \eqref{eq:projAMP} are replaced with the random quantities $\frac{1}{\sqrt{n}\|x_{t-1}\|} w_{ts}$. Here we argue by concentration that the recursion $\eqref{eq:opampProjDenoiser}$ is asymptotically close to a recursion with separable linear denoisers and exact correction coefficient. In particular, for a data matrix $M$, we will compare the recursions
\begin{align}
    x_t & = \delta_t \circ \frac{\sqrt{n}}{\|x_{t-1}\|} \Big( M  x_{t-1} - \frac{1}{n} \sum_{s<t} w_{st} \frac{\sqrt{n}}{\|x_s\|}x_s \Big) + (\mathsf{1}-\delta_t) \circ x_{t-1} \label{eq:opampProjTarget}\\
    \tilde x_t &= \delta_t \circ \frac{1}{\sqrt{\alpha_t}} \Big( M  \tilde x_{t-1} - \frac{1}{n} \sum_{s<t} w_{st} \frac{1}{\sqrt{\alpha_s}} \tilde x_s \Big) + (\mathsf{1}-\delta_t) \circ \tilde x_{t-1} \label{eq:opampProjComparison}
\end{align}   
where the sequence of scalars $\{\alpha_t\}$ is defined recursively in terms of the sequence of measures $\{p_t\}$ and overlaps $\{(q_t,r_t)\}$ generated by SE: 
\begin{align}
    \alpha_t \coloneqq \sum_{s<t} (\lambda^2 r_s^2 + q_s) p_t(s)
\end{align}
Recursion \eqref{eq:opampProjTarget} is just \eqref{eq:opampProjDenoiser}, and \eqref{eq:opampProjComparison} is a valid instance of the projection AMP recursion \eqref{eq:projAMP} with separable denoisers. By convention we set $\frac{\sqrt{n}}{\|x_{-1}\|}x_{-1} \equiv \frac{1}{\sqrt{\alpha_{-1}}}x_{-1}$ to be the same initialization for both recursions, such that Assumption~\ref{as:singLetterAssumptions} holds.

As a first step, we verify that the SE overlaps $\{(q_t,r_t)\}$ associated with \eqref{eq:opampProjComparison} is well-described in the asymptotic limit by the recurrence relation $\{\rho_t\}$ defined by \eqref{eq:PowMethSingleLetter}. For $t=0$, by construction the Gaussian approximation $\tilde y_0$ of $\tilde x_0$ has distribution $\tilde y_0 \sim \normal(\lambda r_0 \theta, \Id_n)$, and it is immediately verified that
\begin{align}
    q_0 \xrightarrow[n \rightarrow \infty]{} 1, \quad
    r_0 \xrightarrow[n \rightarrow \infty]{} \int u \hat u \, \mu(u,\hat u) \eqqcolon \rho_0
\end{align}
Let us assume by way of induction that, for some $t \in \bbN$ and all $s\leq t-1$, 
\begin{align}
    q_s \xrightarrow[n \rightarrow \infty]{} 1, \quad r_s \xrightarrow[n \rightarrow \infty]{} \frac{\lambda \sum_{s'<s} \rho_{s'}p_s(s')}{\sqrt{\sum_{s'<s} (\lambda^2\rho_{s'}^2 + 1) p_s(s')}} \eqqcolon \rho_s
\end{align}
For $t$, then, we study the limits of $(q_t, r_t)$. For $q_t$,
\begin{align}
    q_t & = \frac{1}{n\alpha_t} \ex{\|\tilde y_{t-1}\|^2}  \\
    & = \frac{1}{n \alpha_t} \sum_{i=1}^n \ex{y_{t-1,i}^2}  \\
    & = \frac{1}{n \alpha_t} \sum_{i=1}^n \ex{(\lambda r_{\tau(t,i)} \theta_i + \sqrt{q_{\tau(t,i)}} z_i)^2}  \\
    & = \frac{1}{\alpha_t} \Big( \frac{1}{n}\sum_{i=1}^n \lambda^2 r_{\tau(t,i)}^2 \theta_i^2  + \frac{1}{n} \sum_{i=1}^n q_{\tau(t,i)} \Big)
\end{align}
For each $s<t \in \bbN$, the indicator function $\one\{\tau(t,i)=s\}$ is a measurable function of $\{\delta_{si}\}_{s<t}$,
\begin{align}
    \one\{\tau(t,i)=s\} = \sum_{s<t} \delta_s (1- \delta_{s+1,i}) \dotsm (1-\delta_{t-1,i})
\end{align}
and therefore the summation
\begin{align}
    \frac{1}{n}\sum_{i=1}^n \lambda^2 r_{\tau(t,i)}^2 \theta_i^2 &= \frac{\lambda^2}{n}\sum_{i=1}^n  \Big(\sum_{s<t} r_s^2 \one\{\tau(t,i) = s\} \theta_i^2\Big) 
\end{align}
can be seen as an integral over the empirical measure of $(\theta, \delta_0, \dotsc, \delta_{t-1})$. Then, since by assumption this empirical measure converges to the product measure $\mu \otimes \nu_t$, in the limit the expectation factorizes and we obtain
\begin{align}
    \frac{1}{n}\sum_{i=1}^n \lambda^2 r_{\tau(t,i)}^2 \theta_i^2 \xrightarrow[n \rightarrow \infty]{} \lambda^2 \sum_{s<t} r_s^2 p_t(s) \int u \, d\mu(u, \hat u) = \lambda^2 \sum_{s<t} \rho_s^2 p_t(s)
\end{align}
By a similar argument,
\begin{align}
    \frac{1}{n} \sum_{i=1}^n q_{\tau(t,i)} \xrightarrow[n \rightarrow \infty]{} \sum_{s<t} p_t(s) = 1
\end{align}
Then, by the inductive hypothesis and continuity, we have
\begin{align} \label{eq:qtUnitConvergence}
    q_t = \frac{1}{\sum_{s<t} (\lambda^2 r_s^2 + q_s) p_t(s)} \Big( \frac{1}{n}\sum_{i=1}^n \lambda^2 r_{\tau(t,i)}^2 \theta_i^2  + \frac{1}{n} \sum_{i=1}^n q_{\tau(t,i)} \Big) \xrightarrow[n \rightarrow \infty]{} 1
\end{align}
and
\begin{align}
    r_t = \frac{1}{\sqrt{\sum_{s<t} (\lambda^2 r_s^2 + q_s) p_t(s)}} \Big( \frac{1}{n}\sum_{i=1}^n \lambda r_{\tau(t,i)} \theta_i^2 \Big) \xrightarrow[n \rightarrow\infty]{} \frac{\lambda \sum_{s<t}\rho_s p_t(s)}{\sqrt{\sum_{s<t} (\lambda^2 \rho_s^2 + 1) p_t(s)}} = \rho_t
\end{align}
and induction is complete. This implies that
\begin{align}
    \left| \frac{1}{n} \Big\langle \theta, \frac{\tilde x_t}{\sqrt{\alpha_t}} \Big\rangle - \rho_t \right| \xrightarrow[n \rightarrow \infty]{\mathrm{p}} 0
\end{align}
for all $t \in \bbN_0$, and therefore the separable AMP iterates $\{\tilde x_t\}$ satisfy the convergence result of Theorem~\ref{th:opampPMsingleLetter}. All that is left is to show that, for each $t$, $x_t$ and $\tilde x_t$ are asymptotically close, i.e.
\begin{align}
    \frac{\|x_t - \tilde x_t\|}{\sqrt{n}} \xrightarrow[n \rightarrow]{\mathrm{p}} 0
\end{align}
For $t=0$, the claim is obvious due to the matched initialization. By way of induction, assume this holds until some time $t$. Then, for $t+1$, we get
\begin{align}
    \|x_{t+1} - \tilde x_{t+1}\| &= \Bigg\| \delta_t \circ \Big\{ \frac{\sqrt{n}}{\|x_{t}\|} \Big( M  x_{t} - \frac{1}{n} \sum_{s\leq t} w_{st} \frac{\sqrt{n}}{\|x_s\|}x_s \Big) - \frac{1}{\sqrt{\alpha_{t+1}}} \Big( M  \tilde x_{t} - \frac{1}{n} \sum_{s\leq t} w_{st} \frac{1}{\sqrt{\alpha_s}} \tilde x_s \Big) \Big\} \nonumber \\
    &\qquad + (\mathsf{1} - \delta_t) \circ (x_t - \tilde x_t) \Bigg\| \\
    &\leq \left\| \frac{\sqrt{n}}{\|x_t\|} \Big( Mx_t - \frac{1}{n}\sum_{s \leq t} w_{ts} \frac{\sqrt{n}}{\|x_s\|} x_s \Big) - \frac{1}{\sqrt{\alpha_{t+1}}} \Big( M  \tilde x_{t} - \frac{1}{n} \sum_{s\leq t} w_{st} \frac{1}{\sqrt{\alpha_s}} \tilde x_s \Big) \right\| + \|x_t - \tilde x_t\|  \\
    & \leq \left\| \frac{\sqrt{n}}{\|x_t\|} \Big[\Big( Mx_t - \frac{1}{n}\sum_{s \leq t} w_{ts} \frac{\sqrt{n}}{\|x_s\|} x_s \Big) - \Big( M  \tilde x_{t} - \frac{1}{n} \sum_{s\leq t} w_{st} \frac{1}{\sqrt{\alpha_s}} \tilde x_s \Big)\Big] \right\|  \\
    & \quad + \Big| \frac{\sqrt{n}}{\|x_t\|} - \frac{1}{\sqrt{\alpha_{t+1}}} \Big| \, \left\| M  \tilde x_{t} - \frac{1}{n} \sum_{s\leq t} w_{st} \frac{1}{\sqrt{\alpha_s}} \tilde x_s \right\| + \|x_t - \tilde x_t\|
\end{align}
As a direct consequence of the inductive hypothesis, $\| x_t - \tilde x_t \|/\sqrt{n} \rightarrow 0$ as $n \rightarrow \infty$. Furthermore
\begin{align}
    \Big|\frac{1}{n}\|x_t\|^2 - \frac{1}{n} \E[\|\tilde y_t\|^2]\Big| \xrightarrow[n \rightarrow \infty]{\mathrm{p}} 0
\end{align}
which by continuity and \eqref{eq:qtUnitConvergence} implies that 
\begin{align}
    \Big| \frac{\sqrt{n}}{\|x_t\|} - \frac{1}{\sqrt{\alpha_{t+1}}} \Big| \xrightarrow[n \rightarrow \infty]{\mathrm{p}} 0
\end{align}
Then, since
\begin{align}
    \frac{1}{\sqrt{n}} \Big\| M  \tilde x_{t} - \frac{1}{n} \sum_{s\leq t} w_{st} \frac{1}{\sqrt{\alpha_s}} \tilde x_s \Big\|
\end{align}
is bounded with high probability, we conclude that
\begin{align}
     \Big| \frac{\sqrt{n}}{\|x_t\|} - \frac{1}{\sqrt{\alpha_{t+1}}} \Big| \, \frac{1}{\sqrt{n}}\left\| M  \tilde x_{t} - \frac{1}{n} \sum_{s\leq t} w_{st} \frac{1}{\sqrt{\alpha_s}} \tilde x_s \right\| \xrightarrow[n \rightarrow\infty]{\mathrm{p}} 0
\end{align}
Finally, we use the inductive hypothesis and the fact that all terms $\sqrt{n} / \|x_s\|$ are bounded in probability, for $s\leq t$, to establish
\begin{align}
    &\frac{1}{\sqrt{n}} \left\| \frac{\sqrt{n}}{\|x_t\|} \Big[\Big( Mx_t - \frac{1}{n}\sum_{s \leq t} w_{ts} \frac{\sqrt{n}}{\|x_s\|} x_s \Big) - \Big( M  \tilde x_{t} - \frac{1}{n} \sum_{s\leq t} w_{st} \frac{1}{\sqrt{\alpha_s}} \tilde x_s \Big)\Big] \right\| \\
    &\quad \leq \frac{\sqrt{n}}{\|x_t\|} \, \frac{1}{\sqrt{n}} \left\| M(x_t - \tilde x_t) - \frac{1}{n} \sum_{s \leq t} w_{ts} \Big( \frac{\sqrt{n}}{\|x_s\|}x_s - \frac{1}{\sqrt{\alpha_s}} \tilde x_s\Big) \right\| \\
    &\quad \leq \frac{\sqrt{n}}{\|x_t\|} \frac{1}{\sqrt{n}} \left\{ \|M\|_{op} \|x_t - \tilde x_t\| + \sum_{s \leq t} \frac{w_{ts}}{n} \Big\| \frac{\sqrt{n}}{\|x_s\|}x_s - \frac{1}{\sqrt{\alpha_s}} \tilde x_s \Big\| \right\} \xrightarrow[n \rightarrow \infty]{\mathrm{p}} 0
\end{align}
Thus, $\|x_{t+1} - \tilde x_{t+1}\| / \sqrt{n} \xrightarrow{\mathrm{p}}$ as $n \rightarrow \infty$. This concludes the induction and the proof of Theorem~\ref{th:opampPMsingleLetter}. \qed

\end{document}